\pdfoutput=1
\documentclass[a4paper,USenglish,cleveref,autoref]{lipics-v2021}

\usepackage[utf8]{inputenc} 
\usepackage[T1]{fontenc} 

\title{Online Ramsey numbers: Long versus short cycles}
\author{Grzegorz Adamski}{Faculty of Mathematics and CS, Adam Mickiewicz University, Pozna\'n, Poland}{grzada@amu.edu.pl}{https://orcid.org/0000-0003-0381-0351}{}
\author{Małgorzata {Bednarska-Bzd\c ega}}{Faculty of Mathematics and CS, Adam Mickiewicz University, Pozna\'n, Poland}{mbed@amu.edu.pl}{}{}
\author{Václav Blažej}{Faculty of Information Technology, Czech Technical University in Prague, Prague, Czech Republic}{vaclav.blazej@fit.cvut.cz}{https://orcid.org/0000-0001-9165-6280}{}

\authorrunning{G. Adamski, M. {Bednarska-Bzd\c ega}, and V. Blažej}
\ccsdesc[10003628]{Combinatorial algorithms}

\hideLIPIcs
\renewcommand\blacktriangleleft{\blacksquare}
\renewcommand\vartriangleleft{\varsquare}

\usepackage{amsmath}
\usepackage{amssymb}
\usepackage{graphicx} 
\usepackage{xspace}

\usepackage{amsthm}

\usepackage{tikz}
\usepackage{algorithm}
\usepackage{algpseudocode}
\usepackage{ifthen}
\newcommand{\N}{\mathbb N}

\newcommand{\RR}{\tilde R}
\newcommand{\rr}{\tilde{r}}
\newcommand{\Oh}{\mathcal O}
\newcommand{\IP}{I\!P} 
\newcommand{\dtail}{\mathcal T}
\newcommand{\dragon}{\mathfrak{D}}
\newcommand{\TRI}{\bigtriangleup}
\def\rf{\hat{r}} 

\newcommand\putabove[2]{\mathrel{\overset{\makebox[0pt]{\mbox{\normalfont\tiny\sffamily #2}}}{#1}}}

\usetikzlibrary{calc, shapes, backgrounds, patterns, arrows, decorations, decorations.pathreplacing, decorations.pathmorphing, fit, shapes.misc}
\pgfkeys{%
  /tikz/on layer/.code={
    \pgfonlayer{#1}\begingroup
    \aftergroup\endpgfonlayer
    \aftergroup\endgroup
  }
}
\tikzset{
    >=stealth',
    longpath/.style={decorate, decoration={snake}},
    main/.style={draw=black,fill=white,circle,inner sep=2pt,solid},
    selected/.style={draw,line width=5pt,-,draw opacity=.3,solid,on layer=background,line cap=rect},
    hide/.style = {fill = none, draw = none, shape=rectangle},
    red/.style = {draw = red, solid, semithick},
    blue/.style = {draw = blue, densely dashed, semithick},
    tobe/.style = {draw = black, densely dotted, semithick},
    both/.style = {draw,dashed,blue,dash pattern=on 3pt off 3pt,
    postaction={dashed,red,dash pattern=on 7pt off 11pt, dash phase=-1pt} },
    gray/.style = {draw=gray, very thin},
    black/.style = {draw = black},
    expand bubble/.style={
        preaction={draw,line width=10.4pt},
        white,fill,draw,line width=10pt,
    },
}
\pgfdeclarelayer{background}
\pgfsetlayers{background,main}

\nolinenumbers

\newcommand{\siamver}[2]{#2}

\begin{document}

\maketitle

\begin{abstract}
    Online Ramsey game is played between Builder and Painter on an infinite board $K_{\N}$.
    In every round Builder selects an edge, then Painter colors it red or blue.
    Both know target graphs $H_1$ and $H_2$.
    Builder aims to create either a red copy of $H_1$ or a blue copy of $H_2$ in $K_{\N}$ as soon as possible, and Painter tries to prevent it.
    The online Ramsey number $\rr(H_1,H_2)$ is the minimum number of rounds such that the Builder wins.
    We study $\rr(C_k,C_n)$ where $k$ is fixed and $n$ is large.
    We show that $\rr(C_k,C_n)=2n+\Oh(k)$ for an absolute constant $c$ if $k$ is even, while $\rr(C_k,C_n)\le 3n+o(n)$ if $k$ is odd.
\end{abstract}

\section{Introduction}

The classic Ramsey number $r(H_1,H_2)$ of graphs $H_1,H_2$ is the smallest number of vertices~$n$ such that every $2$-edge colored complete graph on $n$ vertices contains a monochromatic copy of $H_i$ in the $i$-th color, for some $i\in \{1,2\}$.
This presentation naturally leads to a variant where one aims to minimize the number of edges.
Introduced by Erd\H os, Faudree, Rousseau, and Schelp~\cite{Erdos1978TheSRN}, the size-Ramsey number $\rf(H_1,H_2)$ of a pair of graphs $H_1,H_2$ is the smallest integer $m$ such that there exists an $m$-edge graph $G$ with the property that every $2$-edge-coloring of $G$ results in a monochromatic copy of $H_i$ in the $i$-th color.

We study the game variant of the size-Ramsey number called online Ramsey number (or online size-Ramsey number).
A Ramsey game, introduced by Beck \cite{Beck1993AchievementsGA}, is played between Builder and Painter on an infinite board $K_{\N}$, i.e.,~a complete graph on the vertex set $\N$.
In each round, Builder selects an edge and Painter colors it either red or blue.
Given two finite graphs $H_1$ and $H_2$, Builder aims to create either a red copy of $H_1$ or a blue copy of $H_2$ in $K_{\N}$ as soon as possible, while Painter tries to prevent it.
The online Ramsey number $\rr(H_1,H_2)$ is the smallest number of rounds in which Builder can ensure the appearance of a red copy of $H_1$ or a blue copy of $H_2$ as a subgraph, provided both players play optimally.

Although the online Ramsey number is upper bounded by the size-Ramsey number (that is always finite), this bound is often far from optimal.
Probably the most challenging problem in the online Ramsey game theory is to determine $\rr(K_n,K_m)$ asymptotically.
Though exponential bounds on $\rr(K_n,K_n)$ are known \cite{Conlon2009OnlineRN,CFGH}, they differ in the power base, similarly as in the case of the corresponding Ramsey numbers.
Conlon, Fox, Grinshpun, and He \cite{CFGH} studied the off-diagonal case of $\rr(K_n,K_m)$ and, among more general results, proved that $\rr(K_3,K_n)$ is of order $n^3$ up to a polylogarithmic factor.

We focus on the off-diagonal online Ramsey numbers for cycles.
Let us mention first a few results involving cycles in the (off-line) Ramsey theory.

The Ramsey number for a pair of cycles of length $n$ was determined by Bondy and Erd\H os \cite{Bondy1973cyclesRN}.
The work of Rosta \cite{Rosta1973cyclesI,Rosta1973cyclesII} and Faudree and Schelp \cite{Schelp1974cycles} completed the picture by determining the Ramsey number for different combinations of cycle lengths $(C_k,C_n)$.
Their result implies that if $k$ is small and $n$ is big, then $r(C_k,C_n)=n+k/2-1$ for even $k$, while for odd $k$ we have $r(C_k,C_n)=2n-1$.

The size-Ramsey number for two cycles seems much harder to determine.
Haxell, Kohayakawa, and {\L}uczak~\cite{HKL} proved that $\rf(C_k,C_n)=\Oh(n)$ provided $n \ge k \gg \log n$.
Their result is in fact stronger since they considered the induced copies of $C_k$ and $C_n$.
The bounds on $\rf(C_k,C_n)$ in terms of multiplicative constants were later improved several times; recently Javadi, Khoeini, Omidi, and Pokrovskiy \cite{JKOP} showed that if $n$ is large enough and $n\ge k\ge \log n + 15$, then
$\rf(C_k, C_n)\le 113484\cdot 10^6 n$.
Recently, Bednarska-Bzdęga and Łuczak \cite{BL} proved that $\rf(C_k, C_n)\le A\cdot n$ for an absolute constant $A$, when $k$ does not depend on $n$.

As for the online Ramsey number of two cycles, Blažej, Dvořák, and Valla \cite{Blazej2019b} showed how Builder can obtain a monochromatic cycle of length $n$ within at most $72n-3$ rounds (and an induced monochromatic cycle in $735n-27$ rounds).

From the above it is clear that $\rr(C_k, C_n) = \Oh(n)$.
We aim to determine the precise multiplicative constant of $n$ within $\rr(C_k, C_n)$ provided $k \ge 3$ is fixed.
For even $k$ we determine $\rr(C_k, C_n)$ up to an additive term depending on $k$, while for odd $n$ we find a linear upper bound not far from optimal.
The main results of our paper are as follows.

\begin{theorem}\label{thm:even}
    $\rr(C_k,C_n) \le 2n + 20k$ for $n\ge 3k$ and even $k\ge 4$.
\end{theorem}

\begin{theorem}\label{thm:odd}
    $\rr(C_k,C_n) \le 3n + \log_2 n + 50 k$ for $n\ge 8k$ and odd $k\ge 3$.
\end{theorem}

The multiplicative constant $2$ in \Cref{thm:even} is optimal in view of a result by Cyman, Dzido, Lapinskas and Lo \cite{CDLL} who proved that for every connected graph $H$ Painter has a strategy such within $|V(H)|+|E(H)|-1$ rounds no red cycle is created nor a blue copy of $H$.
Hence $\rr(C_k,C_n)\ge 2n-1$ for every $k\ge 3$.
The best known lower bound in case of odd $k$ was proved by Adamski and Bednarska-Bzdęga \cite{lowerbound2021}.
They showed that for every connected graph $H$ Builder needs at least $\varphi |V(H)| + |E(H)|-2\varphi+1$ rounds (where $\varphi \approx 1.618$ denotes the golden ratio) to force Painter to create a red odd cycle or a blue copy of $H$.
In particular, $\rr(C_k,C_n)\ge (\varphi+1)n-2\varphi+1> 2.6n-3$ for any $n$ and odd $k$.

To summarize, the above lower bounds and \Cref{thm:even,thm:odd} imply that
\[
    2n-1\le\rr(C_k,C_n) \le 2n + 20 k \text{\quad for even $k$ and $n\ge 3k$ \quad and}
\]%
\[
    2.6n-3<\rr(C_k,C_n) \le 3n + \log_2 n + 50 k \text{\quad for odd $k$ and $n\ge 8k$.}
\]

The upper bound on $\rr(C_{2k},C_n)$ proves a conjecture posed in \cite{ABC4} that $\rr(C_{2k},P_n)=2n+o(n)$ for every fixed $k\ge 2$.
The exact value of $\rr(C_{2k},P_n)$ remains open for every $k\ge 3$.
The only exact general result is $\rr(C_{4},P_n)=2n-2$ for $n\ge 8$; the lower bound follows from the bound by Cyman, Dzido, Lapinskas, and Lo, while the upper bound comes from \cite{ABC4}.

For odd $k$ we have the following upper bound, slightly better than an immediate consequence of \Cref{thm:odd}.

\begin{theorem}\label{thm:oddpath}
    $\rr(C_k,P_n) \le 3n + 50k$ for every $n$ and odd $k\ge 3$.
\end{theorem}

Thus, in view of the lower bound by Adamski and Bednarska-Bzdęga, for odd $k$ we have
\[
    2.6n-4\le \rr(C_k,P_n)\le 3n + 50k.
\]
For $k=3$ a better upper bound $3n-4$ holds for every $n\ge 3$ \cite{lowerbound2021}.

Our paper is organized as follows.
We present the proofs of the main two \Cref{thm:even,thm:odd} in \Cref{sec:even_buidler_strategy,sec:odd_buidler_strategy}, respectively.
In both sections, we use a general tool for shortening long cycles; this tool is shown in \Cref{sec:approx}.
We do not focus on optimizing the multiplicative constant of $k$ in \Cref{thm:even,thm:odd,thm:oddpath} as minor improvements lead to arguments that are significantly more complex.

\section{Preliminaries}

A graph $H$ is a tuple of vertices $V(H)$ and edges $E(H)$.
Let $v(H)$ and $e(H)$ denote the number of its vertices and edges, respectively.
By $P_m$ we mean a path on $m\ge 1$ vertices.
The \emph{line forest} $L^{(t)}_m$ is a graph on $m\ge 1$ vertices with $t\ge 1$ components where every component is a path.
We denote by ${\mathcal L}^{(\le t)}_m$ the family of all line forests on $m$ vertices with at most $t$ components.

We say a graph is colored if every edge is colored blue or red.
A graph is red if all its edges are red; analogously we define a blue graph.

Let ${\mathcal G}_1$ and ${\mathcal G}_2$ be nonempty families of finite graphs.
The online Ramsey game $\RR({\mathcal G}_1, {\mathcal G}_2)$ is played between Builder and Painter on an infinite board $K_{\mathbb N}$, i.e., an infinite complete graph.
In every round, Builder chooses a previously unselected edge of $K_{\mathbb N}$ and Painter colors it red or blue.
The game ends if after Painter's move there is a red copy of a graph from ${\mathcal G}_1$ or a blue copy of a graph from ${\mathcal G}_2$.
Builder tries to finish the game as soon as possible, while Painter aims to delay Builder's win for as long as possible.
Let $\rr({\mathcal G}_1, {\mathcal G}_2)$ be the minimum number of rounds in the game $\RR({\mathcal G}_1, {\mathcal G}_2)$, provided both players play optimally.
If ${\mathcal G}_i$ consists of one graph $G_i$, we simply write $\rr(G_1, G_2)$ and $\RR(G_1, G_2)$.

Given a colored graph $H$, we also consider a version of the game $\RR({\mathcal G}_1, {\mathcal G}_2)$ where the initial board $K_{\mathbb N}$ already contains $H$.
We denote this game by $\RR_H({\mathcal G}_1, {\mathcal G}_2)$ and the minimum number of rounds Builder needs to achieve his goal in $\RR_H({\mathcal G}_1, {\mathcal G}_2)$ is denoted by $\rr_H({\mathcal G}_1,{\mathcal G}_2)$.
Clearly $\rr_H({\mathcal G}_1,{\mathcal G}_2)\le \rr({\mathcal G}_1,{\mathcal G}_2)$ and $\RR({\mathcal G}_1, {\mathcal G}_2)$ is equivalent to $\RR_H({\mathcal G}_1, {\mathcal G}_2)$ with an empty graph $H$.
Any game $\RR_H({\mathcal G}_1, {\mathcal G}_2)$ will be called shortly a Ramsey game.

Given a partially played of a Ramsey game, by the \emph{host graph} we mean the colored graph induced by the set of all red and blue edges already drawn on the board.
Any vertex that is not incident to any colored edge is called a \emph{free vertex}, and an edge incident to only free vertices is called a \emph{free edge}.

While considering a partially played Ramsey game, we say that Builder can \emph{force a colored graph} $F$ within $t$ rounds, if Builder has a strategy such that after at most $t$ more rounds the host graph contains a copy of $F$.

Throughout the paper we use an idea by Grytczuk, Kierstead, and Pra{\l}at \cite{Grytczuk2008} of forcing two colored paths in a Ramsey game.

\begin{lemma}[\cite{Grytczuk2008}, Theorem~2.3]\label{lem:gryt}
    Let $k \ge 1$.
    In every Ramsey game within $2k - 1$ rounds Builder can force Painter to create two vertex disjoint monochromatic paths: the red one and the blue one, the sum of whose lengths is equal to $k$.
\end{lemma}

The strategy of Builder in the proof of the above lemma is quite simple and can be applied to extending a blue path that is already present in the host graph, provided Painter avoids a red path of a given length.
Thus we obtain the following corollary.

\begin{corollary}\label{cor:gryt2}
Let $k,m \ge 1$ and let $H$ be a blue path with $0\le e(H)< k$.
Assume that in a Ramsey game the host graph contains $H$.
Then Builder has a strategy such that either after $2(k-e(H)+m-1)-1$ rounds the host graph contains a blue path of length $k$ or for some $j<k$ after $2(j-e(H)+m)-1$ rounds the host graph contains two vertex disjoint monochromatic paths: the red one of length $m$ and the blue one of length $j$.
\end{corollary}

All our proofs are based on the following, quite obvious, observation which we use implicitly.

\begin{lemma}[Triangle inequality for Ramsey games]\label{lem:triangleinequality} Let $F$, $G$, and $H$ be (uncolored) graphs, $F_b$ be the blue copy of $F$ and $C$ be any colored graph.
Then
\[
    \rr_C(G,H)\le\rr_C(G,F)+\rr_{F_b}(G,H).
\]
In particular, if $C$ is an empty graph, then
\[
    \rr(G,H)\le\rr(G,F)+\rr_{F_b}(G,H).
\]
\end{lemma}

\section{Approximate cycles}\label{sec:approx}

The key realization of Builder's strategy in $\RR(C_k,C_n)$ presented in further sections will be forcing a long blue cycle, sometimes much longer than $n$.
Therefore, we need a tool to shorten it, provided that Painter never creates a red $C_k$.
The goal of this section is to prove the following theorem.

\begin{theorem}\label{thm:shortenall}
    Let $n>3k-3$, $0\le h\le 2n$, and let $H$ be the blue cycle of length $n+h$.
    Then
    \[
        \rr_{H}(C_k,C_n)\le 3k+10.
    \]
\end{theorem}

We start with the case $h=1$.
\begin{lemma}\label{lem:shortenbluecyclebyone}
    Let $n>3k-3$ and $H$ be a blue cycle of length $n+1$.
    Then
    \[
        \rr_{H}(C_k,C_n)\le k+2.
    \]
\end{lemma}

\begin{proof}
    Let $v_0, v_1, \dots, v_n$ denote the vertices of the blue cycle $H$.
    Let $v_{n+1+x}=v_x$ for every $x$.
    We begin with three easy observations that hold for $q<n$ and any $p$.
    \siamver{\begin{enumerate}}{\begin{enumerate}[(1)]}
        \item\label{short0}
            If a blue edge $v_p v_{p+2}$ is added to $H$, then a blue $C_n$ appears in the host graph.
        \item\label{short1}
            If we add two blue edges $v_p v_{p+q}$ and $v_{p+2} v_{p+q+1}$ to the cycle $H$, then again we have a blue $C_n$.
            See \Cref{fig:observation_short_2}.
        \item\label{short2}
            If a red edge $v_p v_{p+2k-2}$ is added to $H$, then Builder can finish the game within at most $k-1$ rounds by selecting all edges of the path $v_p v_{p+2}v_{p+4}\dots v_{p+2k-2}$.
    \end{enumerate}
    \newcommand{\drawbluepathexample}[1]{
        \pgfmathtruncatemacro\nodeaddon{#1+1}
        \pgfmathtruncatemacro\edgesmone{#1-2}
        \node[hide] (0) at (0,0) {};
        \node[hide] (\nodeaddon) at (\nodeaddon,0) {};
        \foreach \i in {1,...,#1} {
            \node (\i) at (\i,0){};
        }
        \foreach \i in {0,...,#1} {
            \pgfmathtruncatemacro\j{\i+1}
            \draw[blue] (\i) -- (\j);
        }
    }
    \begin{figure}[ht]
        \centering
        \def\nodes{16}
        \begin{tikzpicture}[scale=0.6]
            \tikzstyle{every node}=[main]
            \begin{scope}
                \drawbluepathexample{\nodes}
                \node[hide,label={[shift={(0,-3mm)},rectangle,anchor=north,inner sep=0]$v_{p}$}] at (1) {};
                \node[hide,label={[shift={(0,-3mm)},rectangle,anchor=north,inner sep=0]$v_{p+2}$}] at (3) {};
                \node[hide,label={[shift={(-1mm,-3mm)},rectangle,anchor=north,inner sep=0]$v_{p+q}$}] at (15) {};
                \node[hide,label={[shift={(3mm,-3mm)},rectangle,anchor=north,inner sep=0]$v_{p+q+1}$}] at (16) {};
                \draw[blue] (1) .. controls (2,3) and (14,2) .. (15);
                \draw[blue] (3) .. controls (4,2) and (15,3) .. (16);
                \draw[blue,selected] (1) .. controls (2,3) and (14,2) .. (15);
                \draw[blue,selected] (3) .. controls (4,2) and (15,3) .. (16);
                \draw[blue,selected] (0) to (1);
                \draw[blue,selected] (16) to (17);
                \foreach \i in {3,4,...,14} {
                    \pgfmathtruncatemacro\j{\i+1}
                    \draw[blue,selected] (\i) to (\j);
                }
            \end{scope}
        \end{tikzpicture}
        \caption{
            Example of Observation (\ref{short1}) for $q=14$.
        }%
        \label{fig:observation_short_2}
    \end{figure}
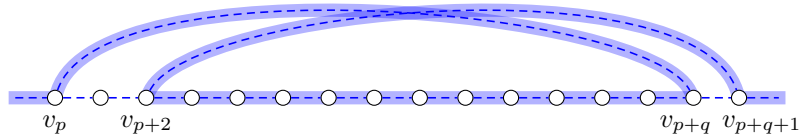

    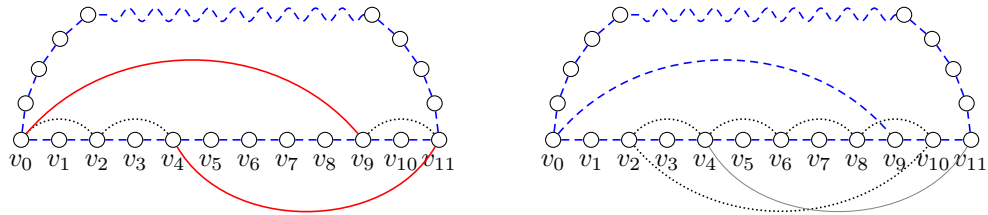
\begin{figure}[ht]
        \centering
        \def\nodes{11}
        \newcommand{\bluecycleexample}{
            \pgfmathtruncatemacro\nodemone{\nodes-1}
            \foreach \i in {0,...,\nodes} {
                \node[label={[shift={(0,-3mm)},rectangle,anchor=north,inner sep=0]$v_{\i}$}] (v\i) at (\i,0){};
            }
            \foreach \i in {0,...,\nodemone} {
                \pgfmathtruncatemacro\j{\i+1}
                \draw[blue] (v\i) -- (v\j);
            }
            \begin{scope}[shift={(7cm,0)}]
                \foreach \i in {1,...,4} {
                    \node (r\i) at (14*\i:4) {};
                }
                \foreach \i in {1,...,3} {
                    \pgfmathtruncatemacro\j{\i+1}
                    \draw[blue] (r\i) -- (r\j);
                }
            \end{scope}
            \begin{scope}[shift={(4cm,0)}]
                \foreach \i in {1,...,4} {
                    \node (l\i) at (180-14*\i:4) {};
                }
                \foreach \i in {1,...,3} {
                    \pgfmathtruncatemacro\j{\i+1}
                    \draw[blue] (l\i) -- (l\j);
                }
            \end{scope}
            \draw[blue] (v\nodes) -- (r1);
            \draw[blue] (v0) -- (l1);
            \draw[blue,longpath] (r4) -- (l4);
        }
        \begin{tikzpicture}[scale=0.5]
            \tikzstyle{every node}=[main]
            \begin{scope}
                \bluecycleexample
                \draw[tobe] (v0) edge[bend left=50] (v2);
                \draw[tobe] (v2) edge[bend left=50] (v4);
                \draw[red]  (v0) edge[bend left=50] (v9);
                \draw[tobe] (v9) edge[bend left=50] (v11);
                \draw[red]  (v4) edge[bend right=60] (v11);
            \end{scope}
            \begin{scope}[shift={(14cm,0)}]
                \bluecycleexample
                \draw[blue]  (v0) edge[bend left=50] (v9);
                \draw[tobe] (v2) edge[bend right=50] (v10);
                \draw[tobe] (v2) edge[bend left=50] (v4);
                \draw[tobe] (v4) edge[bend left=50] (v6);
                \draw[tobe] (v6) edge[bend left=50] (v8);
                \draw[tobe] (v8) edge[bend left=50] (v10);
                \draw[gray] (v4) edge[bend right=60] (v11);
            \end{scope}
        \end{tikzpicture}
        \caption{
            Two cases of Builder's strategy in \Cref{lem:shortenbluecyclebyone} for $k=5$.
            \textbf{Left:} Painter colors both $v_0v_{2k-1}$ and $v_{k-1}v_{3k-4}$ red.
            \textbf{Right:} Painter colors $v_0v_{2k-1}$ blue.
            In both cases Builder chooses the dotted edges.
            If all of them are red, then red $C_k$ appears.
            Otherwise there will be a blue $C_n$.
        }%
        \label{fig:shortenings}
    \end{figure}

    Now we divide the problem based on the parity of $k$.
    \begin{itemize}
        \item $k$ is odd.

        Builder begins the game $\RR_{H} (C_k,C_n)$ with selecting two edges $v_0 v_{2k-1}$ and $v_{k-1} v_{3k-4}$.
        \begin{itemize}
            \item 
        If Painter colors both edges red, Builder can finish the game by selecting all edges of the paths $v_0 v_2 v_4 \dots v_{k-1}$ and $v_{2k-1} v_{2k+1} v_{2k+3} \dots v_{3k-4}$.
        If any of these edges is blue, then there is a blue $C_n$; otherwise there is a red $C_k$.
            \item
        If Painter colors $v_0 v_{2k-1}$ blue, then Builder chooses $v_2 v_{2k}$.
        If Painter colors it blue, the game ends immediately by (\ref{short1}).
        If $v_2 v_{2k}$ is red, Builder can end the game within $k-1$ rounds using (\ref{short2}).
            \item
        In the remaining case, Painter colors $v_{k-1} v_{3k-4}$ blue.
        Then Builder chooses $v_{k-3} v_{3k-5}$ and again, using (\ref{short1}) or (\ref{short2}), he can finish the game within at most $k-1$ rounds.
        \end{itemize}
        \item $k$ is even.

        The proof is very similar to the odd case. Builder begins by selecting edges $v_0 v_{2k-3}$ and $v_{k-2} v_{3k-5}$.
        \begin{itemize}
        \item 
        If Painter colors both edges red, then Builder selects all edges of the paths $v_0 v_2 v_4 \dots v_{k-2}$ and $v_{2k-3} v_{2k-1} \dots v_{3k-5}$.
        If any of these edges is colored blue, then there is a blue $C_n$; otherwise there is a red $C_k$ in the host graph.
        \item 
        If $v_0 v_{2k-3}$ is blue, then Builder chooses $v_1 v_{2k-1}$. If Painter colors it blue, the game ends immediately by (\ref{short1}).
        If $v_1 v_{2k-1}$ is red, Builder can end the game within $k-1$ rounds using (\ref{short2}).
        \item 
        In the remaining case, Painter colors $v_{k-2} v_{3k-5}$ blue.
        Then Builder chooses $v_{k-1} v_{3k-3}$ and again, using (\ref{short1}) or (\ref{short2}), he can finish the game within at most $k-1$ rounds.
        \siamver{}{\qedhere}
        \end{itemize}
    \end{itemize}
\end{proof}

The following lemma will be useful in the proof of \Cref{thm:shortenall} for $h\le k$.

\begin{lemma}\label{lem:shortenbluecyclebyOk}
    Let $h\ge 1$, $k\ge 3$, $n>2k$ and let $H$ be a blue cycle on $n+h$ vertices.
    Then
    \[
        \rr_{H}(C_k,\{C_n,C_{n+1}\})\le h+k+4.
    \]
\end{lemma}
\begin{proof}
    For any integers $p> 2q\ge 0$ we define a colored graph $C(p,q)$ such that it is the union of a blue cycle $v_0v_1\ldots v_{p-1}$ of length $p$ and a red path $v_0 v_2 v_4 \dots v_{2q}$ of length $q$.
    We begin with two easy observations.

    \siamver{\begin{enumerate}}{\begin{enumerate}[(1)]}
        \item\label{Cpq}
            If $p > 2q+2$ and the host graph is $C(p,q)$, then Builder can force $C(p-1,q)$ or $C(p,q+1)$ within $1$ round by selecting edge $v_{2q}v_{2q+2}$.
        \item\label{Cpq2}
            If $p > 2k+2\ell$, $0\le \ell\le k-2$ and the host graph is $C(p,k+\ell)$, then Builder can force $C(p-2-4\ell,k-2-\ell)$ or a red $C_k$ within two rounds by selecting edges $v_0 v_{2k-2}$ and $v_{2\ell+2} v_{2k+2\ell}$.
            This achieves his goal, since coloring both of the edges blue shortens the initial blue cycle by $4\ell+2$ vertices, while the initial red path is shortened by $2\ell+2$ (i.e. the shorter blue cycle has a red path of length $k+\ell-(2\ell+2)=k-\ell-2$ on it); see \Cref{fig:approximate_cycle_cases}.
    \end{enumerate}

    \newcommand{\drawbluepathexample}[1]{
        \pgfmathtruncatemacro\nodeaddon{#1+1}
        \pgfmathtruncatemacro\edgesmone{#1-2}
        \node[hide] (0) at (0,0) {};
        \node[hide] (\nodeaddon) at (\nodeaddon,0) {};
        \foreach \i in {1,...,#1} {
            \node (\i) at (\i,0){};
        }
        \foreach \i in {0,...,#1} {
            \pgfmathtruncatemacro\j{\i+1}
            \draw[blue] (\i) -- (\j);
        }
    }
    \begin{figure}[ht]
        \centering
        \def\nodes{17}
        \begin{tikzpicture}[scale=0.6]
            \tikzstyle{every node}=[main]
            \begin{scope}
                \drawbluepathexample{\nodes}
                \foreach \i in {1,15}{
                    \node[label={[shift={(0,-4mm)},rectangle,anchor=north,inner sep=0]$v_{\i}$}] at (\i) {};
                }
                \foreach \i in {1,3,...,\edgesmone} {
                    \pgfmathtruncatemacro\j{\i+2}
                    \draw[red] (\i) to[bend right=50] (\j);
                }
                \draw[red] (1) .. controls (2,2) and (14,2) .. (15);
                \draw[red,selected] (1) .. controls (2,2) and (14,2) .. (15);
                \foreach \i in {1,3,...,14} {
                    \pgfmathtruncatemacro\j{\i+2}
                    \draw[red,selected] (\i) to[bend right=50] (\j);
                }
            \end{scope}
        \end{tikzpicture}
        \begin{tikzpicture}[scale=0.6]
            \tikzstyle{every node}=[main]
            \begin{scope}
                \drawbluepathexample{\nodes}
                \foreach \i in {1,3,15,17}{
                    \node[label={[shift={(0,-4mm)},rectangle,anchor=north,inner sep=0]$v_{\i}$}] at (\i) {};
                }
                \foreach \i in {1,3,...,\edgesmone} {
                    \pgfmathtruncatemacro\j{\i+2}
                    \draw[red] (\i) to[bend right=50] (\j);
                }
                \foreach \i in {3,5,...,14} {
                    \pgfmathtruncatemacro\j{\i+2}
                    \draw[red,selected] (\i) to[bend right=50] (\j);
                }
                \draw[blue] (1) .. controls (2,3) and (14,2) .. (15);
                \draw[blue] (3) .. controls (4,2) and (16,3) .. (17);
                \draw[blue,selected] (1) .. controls (2,3) and (14,2) .. (15);
                \draw[blue,selected] (3) .. controls (4,2) and (16,3) .. (17);
                \draw[blue,selected] (0) to (1);
                \draw[blue,selected] (17) to (18);
                \foreach \i in {3,4,...,14} {
                    \pgfmathtruncatemacro\j{\i+1}
                    \draw[blue,selected] (\i) to (\j);
                }
            \end{scope}
        \end{tikzpicture}
        \caption{
            Two cases that appear in observation (\ref{Cpq2}) while $k=8$ and $\ell=0$.
            The outlined edges constitute the final structure.
            \textbf{Top:}
            One of the two edges is red.
            \textbf{Bottom:}
            Both of the edges are blue.
        }%
        \label{fig:approximate_cycle_cases}
    \end{figure}
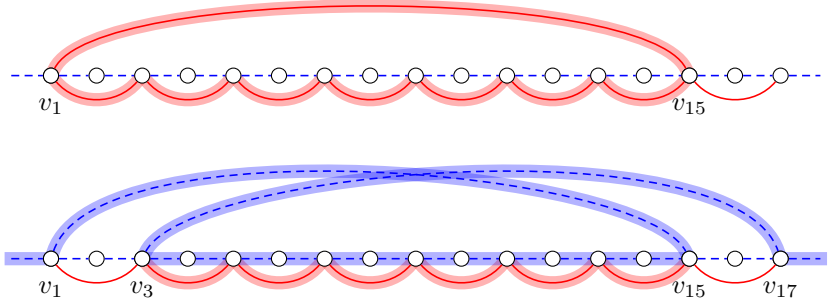

    We will show the strategy for Builder in $\RR_H(C_k,\{C_n,C_{n+1}\})$, which repeatedly uses the above observations.
    First, let us define a potential function of the host graph.
    For $p\ge n$ and $0\le q\le k+1$ we define $f(p,q)=p-n+k-q$ and call it the potential of a colored graph $C(p,q)$.
    Given a colored graph $G$ containing a blue cycle on $n$ vertices, by the potential $f(G)$ of $G$ we mean the smallest potential among all subgraphs $C(p,q)$ contained in $G$, where $p\ge n$ and $0\le q\le k+1$.

    The game starts with the host graph $H=C(n+h,0)$, so $f(H)=h+k$.
    We assume that Painter never creates a red $C_k$.
    Builder's strategy is focused on decreasing the potential after (almost) every round, based on observations (\ref{Cpq}) and (\ref{Cpq2}).
    More precisely, let $C(p,q)$ be the structure with the smallest potential in the host graph after some round of the game.
    As long as $p\ge n+2$, Builder plays in the following way.
    \siamver{\begin{enumerate}}{\begin{enumerate}[(a)]}
        \item
            If $q=k+1$ and $p\ge n+6$, then Builder selects two edges based on (\ref{Cpq2}) applied with $\ell=1$.
            Then the host graph contains a copy of $C(p-6,q-4)$, so the potential decreases by 2, from $f(p,k+1)$ to $f(p-6,k-3)$.
        \item
            If either $q=k$ and $p\ge n+6$, or $q<k$, then Builder selects an edge based on (\ref{Cpq}), so that a copy of $C(p-1,q)$ or $C(p,q+1)$ is created within one round.
            In both cases the potential of the host graph decreases by 1.
        \item\label{ending}
            If $q=k$ and $n+2\le p\le n+5$, then Builder uses (\ref{Cpq2}) applied with $\ell=0$.
            After two rounds he gets a copy of $C(p-2,k-2)$, so the potential of the host graph does not change.
    \end{enumerate}
    The game stops when the host graph $H'$ contains $C(p,q)$ with $p=n$ or $p=n+1$.
    Observe that at this moment $q\le k$, so the potential $f(H')$ is non-negative.
    Note also that case (\ref{ending}) of not changing the potential happens at most twice in the game so there are at most $4$ rounds such that except for these four rounds the potential decreases by one on average per round.
    In view of the starting potential $f(H)=h+k$, we conclude that Builder achieves his goal within at most $h+k+4$ rounds.
\end{proof}

Though \Cref{lem:shortenbluecyclebyOk} holds for every $h$, it gives the upper bound depending on $n$ if $h$ increases with $n$.
The next lemma allows to shorten very long cycles more efficiently.

\begin{lemma}\label{lem:shortenbluecycle}
    Let $k\ge 3$, $k\le h\le 2n$ and let $H$ be a blue cycle of length $n+h$.
    Then
    \[
        \rr_{H}(C_k,\{C_{n+r},C_{n+r+1},C_{n+r+2}\})\le k+1,
    \]
    where $r=0$ if $h$ is even and $r=1$ if $h$ is odd.
\end{lemma}

\begin{proof}
    Let $p=n+h$ and $q=\lfloor h/2\rfloor$. We have $q\ge \lfloor k/2\rfloor$, $p\ge 3q$ and we need to show that 
    \[
        \rr_{H}(C_k,\{C_{p-2q},C_{p-2q+1},C_{p-2q+2}\})\le k+1.
    \]
    Denote the consecutive vertices of $H$ by $v_0,v_1, \dots, v_{p-1}$.
    We present Builder's strategy in two stages.
    The first stage depends on the parity of $k$.
    \begin{itemize}
        \item $k$ is even.

            Builder selects edge $v_0v_{2q}$.
            If Painter colors it blue, then there is a blue $C_{p-2q+1}$ and the game ends.
            Otherwise, we proceed to the second stage.

        \item $k$ is odd.

            Builder chooses edges $v_0 v_{q}$ and $v_{q} v_{2q}$.
            If Painter colors both edges blue, then we have a blue cycle $C_{p-2q+2}$ and the game ends.
            Therefore, we can assume that $v_{q} v_{2q}$ is red.
            Then Builder chooses $v_{2q}v_{3q}$.
            If after Painter's answer both edges $v_0 v_{q}$ and $v_{2q} v_{3q}$ are blue, then we obtain a blue $C_{p-2q+2}$ and finish the game.
            If any of these edges is red, it is adjacent to the red edge $v_{q}v_{2q}$.
            Because of the symmetry, we can assume that the pair of red edges is $v_{0}v_{q}$ and $v_{q}v_{2q}$ and the game proceeds to the second stage.
    \end{itemize}
    In the second stage, the vertices $v_0$ and $v_{2q}$ of the blue cycle $H$ are connected with a red path of length $1$ or $2$ (depending on the parity of $k$).
    Now Builder selects all edges of a path $P_{k}$ (if $k$ is even) or $P_{k-1}$ (if $k$ is odd) from $v_0$ to $v_{2q}$, on the edge set
    \begin{equation*}
        \begin{gathered}
            \bigcup_{i=0}^{\lfloor k/2\rfloor -2} v_i v_{i+2q+1}
            \quad\cup\quad \bigcup_{i=1}^{\lfloor k/2\rfloor -1} v_i v_{i+2q-1}
            \quad\cup\ v_{\lfloor k/2\rfloor -1}v_{2q+\lfloor k/2\rfloor -1},
        \end{gathered}
    \end{equation*}
    see \Cref{fig:forced_odd_red_cycle_or_join2}.
    Every edge $v_x v_y$ of this path satisfies $2q-1 \le |x-y| \le 2q+1$. 

    \begin{figure}[ht]
        \centering
        \begin{tikzpicture}
        \tikzstyle{every node}=[main]
            \begin{scope}
                \def\mypathlength{9}
                \node[hide] at (2,-2) {}; 
                \node[main,label=90:{$v_0$}] (u1) at (1,+1){};
                \node[main,label={[rectangle,shift={(0,-8mm)}]$v_{2q}$}] (v1) at (1,-1){}; 
                \foreach \i in {2,...,\mypathlength}{
                    \pgfmathtruncatemacro\j{\i-1}
                    \node[main,label=90:{$v_{\j}$}] (u\i) at (\i,+1){};
                    \node[main,label={[rectangle,shift={(0,-8mm)}]$v_{2q+\j}$}] (v\i) at (\i,-1){};
                }
                \foreach \i in {2,...,\mypathlength}{
                    \pgfmathtruncatemacro\prev{\i-1}
                    \draw[blue] (u\prev) -- (u\i);
                    \draw[blue] (v\prev) -- (v\i);
                }
                \pgfmathtruncatemacro\lastvtx{\mypathlength+1}
                \node[main,label=180:{$v_q$}] (x) at (-.5,0) {};
                \draw[red] (u1) -- (x) -- (v1);
                \foreach \i in {2,...,\mypathlength}{
                    \pgfmathtruncatemacro\j{\i-1}
                    \draw[red] (u\j) -- (v\i);
                    \draw[red] (v\j) -- (u\i);
                    \draw[red,selected] (u\j) -- (v\i);
                    \draw[red,selected] (v\j) -- (u\i);
                }
                \draw[red] (u\mypathlength) -- (v\mypathlength);
                \draw[red,selected] (u\mypathlength) -- (v\mypathlength);
                \draw[blue] (u1) edge[longpath,bend right] (x);
                \draw[blue,longpath] (x) ..controls +(3,.5) and +(-2,2).. (v\mypathlength);
                \draw[blue,longpath] (v1) ..controls +(1,1) and +(-2,-2).. (u\mypathlength);
            \end{scope}
        \end{tikzpicture}
        \caption{
            The path from $v_0$ to $v_{2q}$ in \Cref{lem:shortenbluecycle} for $k=19$ (outlined), edges $v_0 v_q$ and $v_q v_{2q}$ (thin red) and the blue cycle $C_p$.
        }%
        \label{fig:forced_odd_red_cycle_or_join2}
    \end{figure}
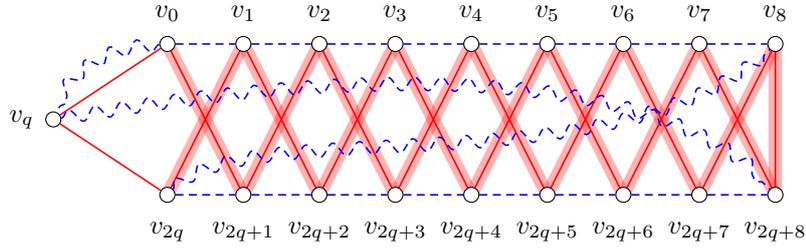
    If Painter avoids a red cycle of length $k$, then one of these edges must be blue.
    Such a blue edge cuts off $2q$, $2q-1$ or $2q-2$ vertices from the cycle $H$ so we obtain a blue cycle on $p-2q$, $p-2q+1$ or $p-2q+2$ vertices in the host graph.

    Finally, notice that Builder selects at most $1+(k-1)$ edges in the game for even $k$, while for odd $k$ the game last at most $3+(k-2)$ rounds.
\end{proof}

\begin{proof}[Proof of \Cref{thm:shortenall}]
If $h\le k$, then the assertion follows from \Cref{lem:shortenbluecyclebyOk} and \Cref{lem:shortenbluecyclebyone}.
Further assume that $k\le h\le 2n$.
We 
apply \Cref{lem:shortenbluecycle}.
It shows that, starting from a blue cycle on $n+h$ vertices, Builder can force either a red $C_k$ or one of the blue cycles $C_{n},C_{n+1},C_{n+2},C_{n+3}$ ($C_{n},C_{n+1}$ or $C_{n+2}$ if $h$ is even; $C_{n+1},C_{n+2}$ or $C_{n+3}$ if $h$ is odd) within $k+1$ rounds.
Then we apply \Cref{lem:shortenbluecyclebyOk} and \Cref{lem:shortenbluecyclebyone} and the assertion of \Cref{thm:shortenall} follows.
In the worst-case scenario, Builder needs $k+1$ rounds to get from the blue $C_{n+h}$ to a blue cycle $C'$ of length at most $n+3$, then $k+7$ rounds to get from $C'$ to $C_{n+1}$, and finally $k+2$ rounds to get a blue $C_n$ from $C_{n+1}$.
\end{proof}

\section{Short even cycle vs long cycle}\label{sec:even_buidler_strategy}

In this section, we show a Builder's strategy for $\rr(C_k,C_n)$ where $k$ is even.
Supporting lemmas will also be used extensively in the odd cycle case in \Cref{sec:odd_buidler_strategy}.
First, let us state overview of the strategy in \Cref{alg:even_case}.

\begin{algorithm}[ht]
    \caption{Builder's strategy for $\rr(C_k,C_n)$ for an even $k$}\label{alg:even_case}%
    \begin{algorithmic}
        \State Build a red $P_{2k-4}$ or a blue $P_{n+\Oh(1)}$ with \Cref{cor:gryt2}.
        \If{we have a red $P_{2k-4}$}
            \State Build $k$ blue paths of total length of $n+\Oh(k)$ with \Cref{lem:icicle}.
            \State Join $k$ paths into two with \Cref{lem:join_to_2_paths}.
            \State Join two paths into one with \Cref{lem:join_2_paths2}.
        \EndIf
        \State Join the blue path into an approximate cycle with \Cref{lem:join_2_paths2}.
        \State Shorten the approximate cycle to $C_n$ using \Cref{thm:shortenall}.
    \end{algorithmic}
\end{algorithm}

A big part of the strategy is spent on forcing a long blue path efficiently using a threat of a red $C_k$.
For this purpose, we now introduce an icicle path, see \Cref{fig:icicle_path}.
\begin{definition}
    \emph{Icicle path} $\IP_{\ell,n}$ is a graph on $n$ vertices that consists of a red path $P_\ell$ on $\ell$ vertices called a \emph{spine} and $\ell$ vertex disjoint blue paths (possibly with no edges) called \emph{icicles}, each having one endpoint in the red path.
\end{definition}

\begin{figure}[ht]
    \centering
    \def\nodes{14}
    \def\nodessum{\nodes}
    \begin{tikzpicture}[scale=0.6]
        \tikzstyle{every node}=[main]
        \begin{scope}
            \pgfmathtruncatemacro\edges{\nodes-1}
            \foreach \i in {1,...,\nodes} {
                \node (\i_0) at (\i,0){};
            }
            \foreach \i in {1,...,\edges} {
                \pgfmathtruncatemacro\j{\i+1}
                \draw[red] (\i_0) to (\j_0);
            }
            \foreach[count=\i] \len in {3,2,1,2,0,2,1,2,1,3,0,0,1,3}{
                \pgfmathtruncatemacro{\nodessum}{\nodessum+\len}
                \xdef\nodessum{\nodessum}
                \ifthenelse{0 = \len}{}{
                    \foreach \r in {1,...,\len}{
                        \pgfmathtruncatemacro\rmone{\r-1}
                        \node (\i_\r) at (\i,-\r) {};
                        \draw[blue] (\i_\rmone) -- (\i_\r);
                    }
                }
            }
        \end{scope}
    \end{tikzpicture}
    \caption{
        Example of an icicle path $\IP_{\nodes,\nodessum}$, i.e., in total there are $\nodessum$ vertices and its red path contains $\nodes$ vertices.
    }%
    \label{fig:icicle_path}
\end{figure}

\begin{lemma}\label{lem:icicle}
    Let $1\le m\le n$ and $k\geq 3$.
    If $B$ is a blue path on $m$ vertices, then
    \[
        \rr_{B}\big(C_k, {\mathcal L}^{(\le k)}_n\big) \le 2n-2m.
    \]
\end{lemma}

\begin{proof}
    Having a blue path $P_m$, Builder will play so that after every round of the game all blue edges of the host graph are contained in a vertex disjoint union of a blue path $P_j$, let us call it a special path, and an icicle path $\IP_{a,t}$.
    We define the score of such a host graph, which changes after every round, as
    \[
        S(\IP_{a,t},P_j) = t+j+\text{number of blue edges of the host graph}.
    \]

    At the start of the game we have an icicle path $\IP_{1,m}$ (the blue path $P_m$) and a free vertex as a special path.
    We claim that Builder can play so that the icicle path has at most $k$ icicles all the time.
    Furthermore, if the icicle path has $k-1$ or $k$ icicles, then the special path is trivial; moreover if the icicle path has $k$ icicles, then its last icicle is trivial.
    We also claim that Builder can increase the score of the host graph by one in every round provided no red $C_k$ is created.

    Suppose that our statement is true after round $r\ge 0$, the host graph has an icicle path $\IP_{a,t}$ and a blue path $P_j$.
    Denote vertices of the spine of the icicle path by $v_1,v_2,\dots,v_a$ and let $b_i$ be the number of vertices of the blue icicle with an endpoint $v_i$.
    The score now is
    \[
        S(\IP_{a,t},P_j) = t+j+\Big(j-1+\sum_{i=1}^a (b_i-1)\Big) = 2t+2j-a-1.
    \]

    We claim that Builder can increase the score by one in round $r+1$ by proceeding with the following strategy.
    Denote one endpoint of the special path $P_j$ by $u$.
    We consider three cases, depending on the length of the spine.
    \begin{enumerate}

        \item\label{case:new_path}
            If $a \leq k-2$, then Builder selects $v_a u$.

            If Painter colors it red, then $v_a u$ increases the icicle path to $\IP_{a+1,t+1}$, we take any free vertex as a new (empty) special path, and the score is increased by $1$.
            Observe that if the spine of the new icicle path has $k-1$ vertices, then the new special path is trivial as required.

            If Painter colors $v_a u$ blue, then we forget about the red edge $v_{a-1}v_a$, the path $P_j$ is increased by the blue icicle with the endpoint $v_a$ and we obtain the blue path on $j+b(v_a)$ vertices, while the new icicle path is $\IP_{a-1,t-b(v_a)}$.
            The new blue edge increases the score by $1$.

        \item
            If $a=k-1$, then Builder selects $v_a u$.

            Recall that by the inductive assumption the special path is trivial.
            If Painter colors $v_a u$ red, then $v_a u$ increases the icicle path to $\IP_{k,t+1}$, we take any free vertex as a new (empty) special path, and the score is increased by $1$.
            The last blue icicle of $\IP_{k,t+1}$ is trivial and the new special path is trivial as required.

            If Painter colors $v_a u$ blue, then we forget about the red edge $v_{a-1}v_a$, the path $P_j$ is increased by the blue icicle path with the endpoint $v_a$ and we obtain the blue path on $j+b(v_a)$ vertices, while the new icicle path is $\IP_{a-1,t-b(v_a)}$.
            The new blue edge increases the score by $1$.

        \item
            If $a = k$, then Builder selects the edge $v_k v_1$.

            By the inductive assumption the special path is trivial and the icicle with the endpoint $v_k$ is trivial.
            If Painter colors $v_k v_1$ red, then $C_k$ is created and Builder wins.
            If Painter colors it blue, then we forget about the red edge $v_1v_2$, the icicle with the endpoint $v_1$ gets moved to $v_k$ and thereby we obtain a new icicle path $\IP_{k-1,t}$ with one blue edge more than in $\IP_{k,t}$.
            The score is increased by $1$, the special path remains trivial.
    \end{enumerate}
    In all cases the statement is true after round $r+1$.

    The starting position with a blue path $P_m$ and a trivial special path $u$ has the score $(\IP_{1,m},u) =2m$.
    As long as there is no red $C_k$, the score increases after every round so after $2n-2m$ rounds it reaches $2n$.
    Then we have an icicle path $\IP_{a,t}$ and a special path $P_j$ such that they contain all blue edges of the host graph and $a\le k$.
    Thus $2n=S(\IP_{a,t},P_j) = 2(t+j)-a-1$ and hence $t+j=(2n+a+1)/2\ge n+1$.
    We infer that $\IP_{a,t}\cup P_j$ contains at most $k$ vertex disjoint blue paths of total number of vertices at least $n$ (in case $a=k$ we ignore the special path, which is trivial).
\end{proof}

\begin{observation}\label{obs:avoid_ck}
    If there is a red path $P_{k-1}$ where $v_1$ and $v_{k-1}$ are its endpoints, then if Builder selects edges $v_1x$ and $v_{k-1}x$ where $x \not\in V(P_{k-1})$, then Painter must color one of these edges blue if he wants to avoid a red $C_k$.
\end{observation}

In the next lemma we show that, having a red path on $k+t-4$ vertices, Builder can force Painter to connect $t$ disjoint paths into two paths.
The parity of $k$ is irrelevant.

\begin{lemma}\label{lem:join_to_2_paths}
    Suppose that $k\geq 3$, $t\ge 2$ and $H$ is a colored graph consisting of two vertex disjoint graphs: a blue $L^{(t)}_m$ and a red path of length $k+t-4$.
    Then
    \[
        \rr_{H}\big(C_k, {\mathcal L}^{(\le 2)}_{m+t-2}\big) \le 5(t-2).
    \]
\end{lemma}
\begin{proof}
    Let $v_1,\dots,v_{k+t-4}$ denote vertices of the red path.
    Notice that \Cref{obs:avoid_ck} can be used between every pair of vertices $v_i$ and $v_{i+k-2}$, for $i\le t-2$.

    Starting with $t$ blue paths on $m$ vertices in total, Builder repeats $t-2$ times the following procedure which joins two blue paths.
    Take three blue paths $P$, $P'$, $P''$ on $b_1$, $b_2$, $b_3$ vertices, respectively, and let $x_1$, $x_2$, $x_3$ be one of their endpoints, respectively.
    If this is the $i$-th iteration of this procedure, then Builder selects edges $x_1v_i$, $x_2v_i$, and $x_3v_i$.
    If Painter colored two of them blue, then Builder achieved his goal.
    Otherwise, there are two red edges, say $x_1v_i$ and $x_2v_i$.
    Builder selects $x_1v_{i+k-2}$ and $x_2v_{i+k-2}$.
    Both edges must be blue by \Cref{obs:avoid_ck} so there is a blue path on $b_1+b_2+1$ vertices that joins $P$ and $P'$ through either $v_i$ or $v_{i+k-2}$.

    Every five rounds Builder either forces a red $C_k$ or he merges two of the blue paths through an additional vertex.
    Thus within $5(t-2)$ rounds Builder obtains two blue paths with the sum of their vertex numbers $m+t-2$ (or a red $C_k$).
\end{proof}

\begin{lemma}\label{lem:join_2_paths}
    Let $k\geq 4$ be an even integer.
    Suppose that there are two blue vertex disjoint paths $P$ and $Q$ both on $k/2$ vertices and $p,q$ are one of their endpoints, respectively.
    Then within $k$ rounds Builder can force a red $C_k$ or a blue path $R$ on at least $k/2$ vertices of $P\cup Q$, with its endpoints $p$ and $q$.
\end{lemma}
\begin{proof}
    Let $u_1,\dots,u_{k/2}$ denote the vertices in $P$ and let $v_1,\dots,v_{k/2}$ denote the vertices in $Q$.
    Builder selects edge $u_1v_1$, all edges $u_iv_{k/2-i+1}$ for $i \in [k/2]$, as well as $u_iv_{k/2-i+2}$ for each $2\le i \le k/2$; see \Cref{fig:cycle_knot}.

    These edges constitute a cycle on $k$ vertices so if Painter colors all the edges red, then Builder obtains a red $C_k$; otherwise, the blue paths are joined by a blue edge.
    The blue edge connects the paths in a way that the resulting blue path $R$ with endpoints $u_{k/2}, v_{k/2}$ skips $k/2$ vertices in the worst case.
    Hence, the final path $R$ contains at least $k/2$ vertices.
    \begin{figure}[ht]
        \centering
        \begin{tikzpicture}
            \def\maxlen{5}
            \pgfmathtruncatemacro\maxlenminusone{\maxlen-1}
            \pgfmathtruncatemacro\maxlenplusone{\maxlen+1}
            \node at (0,-0) {$P$};
            \node at (0,-2) {$Q$};
            \foreach \x in {1,...,\maxlen} {
                \node[main,label=+90:{$u_{\x}$}] (u\x) at (\x,-0){};
                \node[main,label=-90:{$v_{\x}$}] (v\x) at (\x,-2){};
            }
            \node[main,label=+90:{$p$}] (uend) at (\maxlenplusone,-0) {};
            \node[main,label=-90:{$q$}] (vend) at (\maxlenplusone,-2) {};
            \draw[blue] (uend)--(u\maxlen);
            \draw[blue] (vend)--(v\maxlen);
            \draw[blue,selected] (uend)--(u5)--(u4) edge[bend left=2] (v3) (v3)--(v4)--(v5)--(vend);
            \foreach \x in {1,...,\maxlenminusone} {
                \pgfmathtruncatemacro\xx{\x+1}
                \draw[blue] (u\x) -- (u\xx);
                \draw[blue] (v\x) -- (v\xx);
            }
            \foreach \x in {1,...,\maxlen} {
                \pgfmathtruncatemacro\xx{\maxlen+1-\x}
                \draw[red] (u\x) edge[bend right=2] (v\xx);
            }
            \foreach \x in {2,...,\maxlen} {
                \pgfmathtruncatemacro\xx{\maxlen+2-\x}
                \draw[red] (u\x) edge[bend left=2] (v\xx);
            }
            \draw[red] (u1) -- (v1);
        \end{tikzpicture}
        \caption{
            Cycle created to join two partial paths for $\frac k2=5$.
            A possible resulting path $R$ that would be created if $u_4v_3$ was blue is outlined.
        }%
        \label{fig:cycle_knot}
    \end{figure}
\end{proof}

Typically, we use \Cref{lem:join_2_paths} to join two long blue paths or joining endpoints of a single path to form a cycle while losing at most $k/2$ vertices in the process.
In most cases the following simpler version of this lemma would be sufficient for our needs.

\begin{lemma}\label{lem:join_2_paths2}
    Let $k\geq 4$ be an even integer and let $H$ be a blue line forest on $m>k/2$ vertices, with two components.
    Then $\rr_H(C_k,P_{m-k/2})\le k$.
\end{lemma}

\begin{proof}
    Let $H_1$ and $H_2$ be the blue components of $H$.
    If one of the paths $H_1$, $H_2$ has less than $k/2$ vertices, then the other contains a blue path on $m-k/2$ vertices and Builder wins the game, doing nothing.
    Thus assume that $v(H_1),v(H_2)\ge k/2$.
    Let $P$, $Q$ be the blue paths induced on last $k/2$ vertices of $H_1$ and $H_2$, respectively.
    We apply \Cref{lem:join_2_paths} to $P$, $Q$ and obtain either a red $C_k$ or a blue path on at least $v(H_1)+v(H_2)-k/2=m-k/2$ vertices within $k$ rounds.
\end{proof}

\begin{theorem}\label{thm:cyclepath}
     If $k\geq 4$ is even and $n\in{\mathbb N}$, then $\rr(C_k,P_n) \le 2n+11k$.
\end{theorem}

\begin{proof}

    First, Builder applies a strategy from \Cref{cor:gryt2} until a red path $P_{2k-4}$ or a blue path $P_n$ is created.
    If we got blue $P_n$, the game ends after at most $2(n-1+(2k-6))-1=2n+4k-15$ rounds.
    Otherwise, for some $m<n$ after at most $2(2k-5+m-1)-1=2m+4k-13$ rounds there are vertex disjoint paths: a red path $P_{2n-4}$ and a blue path $P_m$ in the host graph.
    Further we assume that the latter case holds.
    We also assume that Painter never creates a red $C_k$.

    In the next stage of the game Builder forces a blue line forest $L^{(t)}_{n+k/2}$ with $t\le k$, disjoint from the red $P_{2k-4}$.
    In view of \Cref{lem:icicle}, this stage lasts at most $\rr_{P_m}(C_k, {\mathcal L}^{(\le k)}_{n+k/2}) \le 2(n+k/2)-2m=2n-2m+k$ rounds.

    Now we have a blue $L^{(t)}_{n+k/2}$ and the red path $P_{2k-4}$.
    \Cref{lem:join_to_2_paths} ensures that within next $5(t-2)$ rounds Builder can force Painter to create a blue line forest $L$ on $t-2+n+k/2\ge n+k/2$ vertices, with at most two components.

    Finally, it follows from \Cref{lem:join_2_paths2} that within next $\rr_L(C_k,P_n)=k$ rounds Builder can force a blue path on $n$ vertices.

    In the worst case scenario, the game $\RR(C_k,P_n)$ lasts not longer than
    \[
        2m+4k-13+2n-2m+k+5(k-2)+k<2n+11k
    \]
    rounds.
\end{proof}

\begin{proof}[Proof of \Cref{thm:even}]
    Let $k\ge 4$ be an even integer and $n\ge 3k$.
    We assume that during the game Painter never creates a red $C_k$.
    In the first stage of the game Builder forces a blue path $P$ on $n+k/2$ vertices.
    Based on \Cref{thm:cyclepath} he can achieve this goal within $2n+k+11k=2n+12k$ rounds.
    Then, since $n>k/2$, we apply \Cref{lem:join_2_paths} to the ending segments of $P$ and conclude that Builder can force a blue cycle on at least $n$ vertices of $P$, within $k$ rounds.
    The blue cycle has $n+h$ vertices with $0\le h\le k/2$ and finally we use \Cref{thm:shortenall} to shorten it into a cycle $C_n$ within $3k+10$ rounds.
    The total number of rounds is at most $2n+12k+k+2k+10< 2n+20k$.
\end{proof}

\section{Short odd cycle vs long cycle}\label{sec:odd_buidler_strategy}

Compared to the even case, odd case is not easily put into a clear pseudo-code.
Hence, along with the formal description we present a diagram of the strategy in \Cref{fig:stagesCases} that shows how the complex parts of the algorithm fit into each other.

Builder's strategy in $\RR(C_k,C_n)$ is based on wish triangles.
This structure may appear during the construction and helps with achieving the goal.

\begin{definition}\label{def:wish_triangle}
    Let a \emph{wish triangle} be a colored $C_3$ with either three red edges or one red and two blue edges; see \Cref{fig:wish_triangles}.
\end{definition}

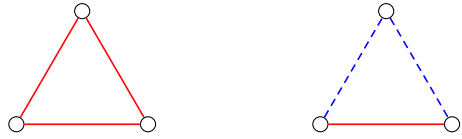
\begin{figure}[ht]
    \centering
    \begin{tikzpicture}
        \begin{scope}[shift={(0cm,0)}]
            \node[main] (a) at (90:1) {};
            \node[main] (b) at (210:1) {};
            \node[main] (c) at (330:1) {};
            \draw[red] (a) -- (b) -- (c) -- (a);
        \end{scope}
        \begin{scope}[shift={(4cm,0)}]
            \node[main] (a) at (90:1) {};
            \node[main] (b) at (210:1) {};
            \node[main] (c) at (330:1) {};
            \draw[blue] (b) -- (a) -- (c);
            \draw[red] (b) -- (c);
        \end{scope}
    \end{tikzpicture}
    \caption{The wish triangles}%
    \label{fig:wish_triangles}
\end{figure}

Before presenting Builder's strategy, we prove several auxiliary lemmata.
The first lemma of this section is similar to \Cref{lem:join_2_paths}, however, due to $k$ being odd we cannot use the threat of a red $C_k$ in the same way.
We show that a similar construction as in the proof of \Cref{lem:join_2_paths} works as long as we have a red path on three vertices connected to the blue paths in advance.

\newcommand{\wishtriangle}[1]{ 
    \node[main,label=-90:{$t_2$}] (t2) at (-1.3,0) {};
    \node[main,label=+90:{$t_1$}] (t1) at ($(t2)+(+40:1.3)$) {};
    \node[main,label=-90:{$t_3$}] (t3) at ($(t2)+(-40:1.3)$) {};
    \draw[#1] (t1) -- (t2) -- (t3);
    \draw[red] (t1) -- (t3);
}
\newcommand{\buildexample}[1]{ 
    \foreach \i in {1,...,#1}{
        \node[main,label=+90:{$u_\i$}] (u\i) at (\i,+1){};
        \node[main,label=-90:{$v_\i$}] (v\i) at (\i,-1){};
    }
    \foreach \i in {2,...,#1}{
        \pgfmathtruncatemacro\prev{\i-1}
        \draw[blue] (u\prev) -- (u\i);
        \draw[blue] (v\prev) -- (v\i);
    }
    \pgfmathtruncatemacro\lastvtx{#1+1}
    \node (uend) at (\lastvtx,+1) {};
    \node (vend) at (\lastvtx,-1) {};
    \draw[blue] (uend)--(u#1);
    \draw[blue] (vend)--(v#1);
}

\begin{lemma}\label{lem:red_gluepath}
    Let $k\ge 3$ be an odd integer.
    Suppose that a colored graph $H$ contains two blue vertex disjoint paths $P$ and $Q$ both on $\lfloor k/2 \rfloor$ vertices and a vertex $x\notin V(P\cup Q)$.
    Let $p,p'$ and $q,q'$ be endpoints of $P$ and $Q$, respectively.
    Suppose that the edges $xp'$, $xq'$ are red.
    Then within $k-2$ rounds Builder can force a red $C_k$ or a blue path $R$ on at least $\lceil k/2\rceil$ vertices of $P\cup Q$, with its endpoints $p$ and $q$.
\end{lemma}
\begin{proof}
    As the endpoints of paths $P$ and $Q$ are connected by a red path of length $2$ Builder may use the strategy from the proof of \Cref{lem:join_2_paths} while substituting one of its red edges by the red path through $x$.
    More precisely, Builder selects edges
    \begin{align*}
        u_iv_{\lfloor\frac k2\rfloor-i+1}     & \quad\text{for all $i \leq \Big\lfloor\frac k2\Big\rfloor$ and}\\
        u_{i+1}v_{\lfloor\frac k2\rfloor-i+1} & \quad\text{for all $i \leq \Big\lfloor\frac k2\Big\rfloor-1$.}
    \end{align*}
    \begin{figure}[ht]
        \centering
        \begin{tikzpicture}
            \begin{scope}
                \def\mypathlength{6}
                \buildexample{\mypathlength}
                \node[main,label=180:{$x$}] (x) at (-0.5,0) {};
                \draw[red] (u1) -- (x) -- (v1);
                \node at (0.2,+1) {$P$};
                \node at (0.2,-1) {$Q$};
                \foreach \i in {1,...,\mypathlength}{
                    \pgfmathtruncatemacro\j{\mypathlength+1-\i}
                    \draw[red] (u\i) edge[bend right=2] (v\j);
                }
                \foreach \i in {2,...,\mypathlength}{
                    \pgfmathtruncatemacro\j{\mypathlength+2-\i}
                    \draw[red] (u\i) edge[bend left=2] (v\j);
                }
                \draw[blue,selected] (uend) -- (u6) -- (u5) -- (u4) -- (u3) edge[bend left=2] (v5) (v5) -- (v6) -- (vend);
            \end{scope}
        \end{tikzpicture}
        \caption{
            A red $C_{13}$ created by a Builder using strategy of \Cref{lem:join_2_paths} while assuming he had red connections $xu_1$ and $xv_1$ to start with.
        }%
        \label{fig:forced_odd_red_cycle_or_join}
    \end{figure}
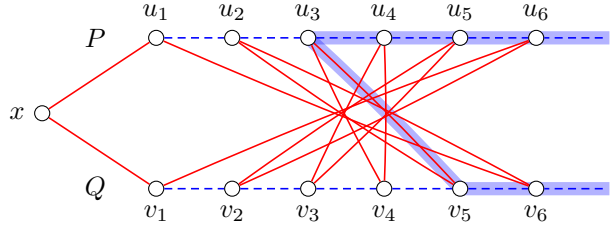

    With these $k-2$ edges Builder either creates a red cycle of length $k$ or he joins the blue paths as desired; see \Cref{fig:forced_odd_red_cycle_or_join}.
\end{proof}

We now show that we are able to join two blue paths if we have a wish triangle in hand.

\begin{lemma}\label{lem:gluepath}
    Let $k\ge 3$ be an odd integer.
    Suppose that the graph $H$ contains three vertex disjoint graphs: two blue paths $P$ and $Q$ both on at $\lfloor k/2\rfloor$ vertices and a wish triangle.
    Let $p$ and $q$ be one of the endpoints of $P$ and $Q$, respectively.
    Then within $k+4$ rounds Builder can force a red $C_k$ or a blue path $R$ on at least $\lceil k/2\rceil$ vertices of $P\cup Q$, with its endpoints $p$ and $q$.
\end{lemma}
\begin{proof}
    Let us denote the wish triangle vertices by $t_1$, $t_2$, and $t_3$ such that $t_1t_3$ is a red edge.
    Denote the vertices of $P$ by $u_1,u_2,\ldots,u_s$, the vertices of $Q$ by $v_1,v_2,\ldots,v_t$ and suppose that $u_s$ and $v_t$ will be the endpoints of the target path $R$.

    First, Builder selects edges $u_1t_1$, $v_1t_1$, $u_2t_3$, and $v_2t_3$.
    If both $u_1t_1$ and $v_1t_1$ are blue, then the blue paths are joined into a blue path on $k+1$ vertices so we reached the goal.
    If both $u_1t_1$ and $v_1t_1$ are red, then Builder uses \Cref{lem:red_gluepath} to join the paths within $k$ rounds, see \Cref{fig:two_same_edges_case}.
    Similarly, we argue the case where $u_2t_3$ and $v_2t_3$ are both blue or both red.

    \begin{figure}[ht]
        \centering
        \begin{tikzpicture}[yscale=.7]
            \begin{scope}[shift={(0cm,0)}]
                \node[main,label=180:{$x$}] (x) at (0,0) {};
                \buildexample{2}
                \draw[blue] (u1) -- (x) -- (v1);
            \end{scope}
            \begin{scope}[shift={(5.5cm,0)}]
                \node[main,label=180:{$x$}] (x) at (0,0) {};
                \buildexample{2}
                \draw[red] (u1) -- (x) -- (v1);
                \node at (1.8,0) {use \Cref{lem:red_gluepath}};
            \end{scope}
        \end{tikzpicture}
        \caption{
            Case analysis for the first four edges.
            Vertex $x$ represents either $t_1$ or $t_3$.
        }%
        \label{fig:two_same_edges_case}
    \end{figure}
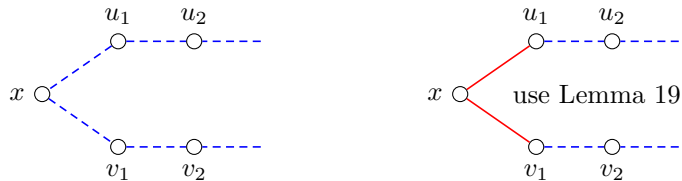

    If none of the previous cases occurs, then the new edges from $t_1$ have distinct colors, and similarly for $t_3$.
    We may assume, without loss of generality by swapping $P$ and $Q$, that $t_1u_1$ is blue and $t_1v_1$ is red.
    Now we distinguish two cases depending on colors of edges from $t_3$.

    If the edge $t_3u_1$ is red and $t_3v_1$ is blue, then depending on the wish triangle colors we have two subcases, depicted on \Cref{fig:cycle_knot2}.
    If the wish triangle is blue, then we are done because the blue paths $P$ and $Q$ are joined through $(u_1,t_1,t_2,t_3,v_1)$ blue path.
    If the wish triangle is red, then we may use the strategy of \Cref{lem:red_gluepath} because $t_1$ and $t_3$ can be thought of as parts of the blue paths and they are joined with a red path of length $2$.

    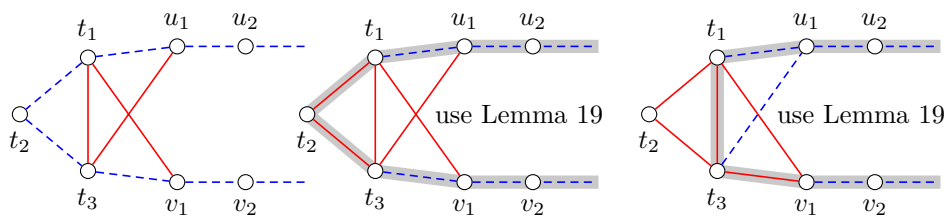
\begin{figure}[ht]
        \centering
        \begin{tikzpicture}[scale=.9]
            \begin{scope}[shift={(0cm,0)}]
                \wishtriangle{blue}
                \buildexample{2}
                \draw[blue] (u1) -- (t1) (t3) -- (v1);
                \draw[red] (u1) -- (t3) (t1) -- (v1);
            \end{scope}
            \begin{scope}[shift={(4.2cm,0)}]
                \wishtriangle{red}
                \buildexample{2}
                \draw[blue] (u1) -- (t1) (t3) -- (v1);
                \draw[red] (u1) -- (t3) (t1) -- (v1);
                \node at (1.8,0) {use \Cref{lem:red_gluepath}};
                \draw[black!70,selected] (uend) -- (u2) -- (u1) -- (t1) -- (t2) -- (t3) -- (v1) -- (v2) -- (vend);
            \end{scope}
            \begin{scope}[shift={(9.2cm,0)}]
                \wishtriangle{red}
                \buildexample{2}
                \draw[blue] (u1) -- (t1) (t3) -- (u1);
                \draw[red] (v1) -- (t3) (t1) -- (v1);
                \node at (1.8,0) {use \Cref{lem:red_gluepath}};
                \draw[black!70,selected] (uend) -- (u2) -- (u1) -- (t1) -- (t3) -- (v1) -- (v2) -- (vend);
            \end{scope}
        \end{tikzpicture}
        \caption{
            Case analysis if blue edges from $t_1$ and $t_3$ end up in either different (left and middle) or the same vertices (right).
            \textbf{Left:}
            The wish path $(t_1,t_2,t_3)$ is blue -- the blue paths get connected.
            \textbf{Middle:}
            The wish path is red -- use \Cref{lem:red_gluepath} on the outlined subgraph.
            \textbf{Right:}
            Regardless of the wish triangle colors we use \Cref{lem:red_gluepath} on the outlined subgraph.
        }%
        \label{fig:cycle_knot2}
    \end{figure}

    If the edge $t_3u_1$ is blue and $t_3v_1$ is red, then we have a single case that does not depend on the color of the wish triangle, depicted on \Cref{fig:cycle_knot2}.
    We can use \Cref{lem:red_gluepath} on the blue paths because one got extended by $u_1t_1$ and they are connected by a red path $t_1t_3v_1$.
    This covers the remaining cases and concludes the proof.
\end{proof}

The following two lemmata are corollaries of the previous one.
We omit their proofs since the first one is almost the same as the proof of \Cref{lem:join_2_paths2}, while \Cref{lem:gluepath2cycle} follows almost immediately from \Cref{lem:gluepath}.

\begin{lemma}\label{lem:gluepath2}
    Let $k\geq 3$ be an odd integer and $m>\lfloor k/2\rfloor$.
    Suppose that a colored graph $H$ contains two vertex disjoint graphs: a blue $L^{(2)}_m$ and a wish triangle.
    Then $\rr_H(C_k,P_{m-\lfloor k/2\rfloor})\le k+4$.
\end{lemma}

\begin{lemma}\label{lem:gluepath2cycle}
Let $k\geq 3$ be an odd integer, $m\ge k$ and let ${\mathcal C}$ be the family of all cycles of length at least $m-\lfloor k/2\rfloor$ and not greater than $m$.
Suppose that a colored graph $H$ contains two vertex disjoint graphs: a blue path on $m$ vertices and a wish triangle.
Then $\rr_H(C_k,{\mathcal C})\le k+4$.
\end{lemma}

Although wish triangles help to glue blue paths together, they are not sufficient.
We now present lemmata that gradually lead to the full proof.
Know that our main goal is to produce a long blue path and proceed further in a similar way as in the even case.
To get there, we either use wish triangles or exploit the fact that they cannot be created.
Throughout the procedure, the longest blue path may be little broken which is covered by the following definition.

\begin{definition}
    Let an \emph{almost blue path} be a path that has at most one red edge and its remaining edges are blue.
\end{definition}

\begin{lemma}\label{lem:longerpath}
    Suppose the colored graph $F$ consists of two components: an almost blue path on $t \ge 1$ vertices and a wish triangle.
    Let $s \ge 1$.
    Then $\rr_F(C_k,P_{t+s})\le 2s+12k$.
\end{lemma}
\begin{proof}
    We assume that Painter never creates a red $C_k$.
    Let $T$ be the wish triangle in $F$.
    If the almost blue path $F\setminus T$ contains a red edge, then we denote by $H$, $H'$ two maximal blue path contained in $F\setminus T$; otherwise we put $H=F\setminus T$ and make a technical assumption that $v(H')=0$.
    We divide the game into three stages.

    In the first stage, Builder applies a strategy from \Cref{cor:gryt2}, extending $H$ or creating a new red path.
    This results in a red path $P_{2k-4}$ or a blue path $P_{v(H)+s+k}$.
    If we got a blue path on $v(H)+s+k$ vertices, the game proceeds to the third stage (the second stage is omitted) and then the first stage lasts at most $2(s+k+2k-6)-1<2s+6k$ rounds.
    Otherwise, for some $m<s+k$ after at most $2(m+2k-5)-1<2m+4k$ rounds new vertex disjoint paths appear in the host graph: a red path $P'$ on $2k-4$ vertices and a blue path $P$ on $v(H)+m$ vertices.
    The game proceeds to the second stage.

    In the second, stage Builder pretends he plays the game $\RR_P(C_k, {\mathcal L}^{(\le k)}_{t'})$ with $t'=t-v(H')+s+k$, on the board vertex disjoint from $(V(F)\setminus V(H))\cup V(P)$.
    Based on \Cref{lem:icicle}, after at most $2t'-2v(P)=2(t-v(H')+s+k)-2(v(H)+m)=2s+2k-2m$ rounds of the second stage we have a blue line forest on $t-v(H')+s+k$ vertices, with at most $k$ blue components.
    This line forest together with $H'$ (if any) form a blue line forest $L$ with at most $k+1$ components, on $t+s+k$ vertices.
    Let us recall that the host graph contains also the red path $P'$ on $2k-4$ vertices.
    Next, we use \Cref{lem:join_to_2_paths} to join the components of $L$ and to obtain two blue paths on $t+s+k$ vertices within $\rr_L(C_k,{\mathcal L}^{(\le 2)}_{t+s+k})\le 5(t-2)\le 5(k-1)$ rounds.
    The second stage last at most $2s+2k-2m+5(k-1)<2s+7k-2m-5$ rounds.

    The third stage begins after less than $2s+6k$ rounds of the first stage if $v(H')\neq 0$ or after less than $(2m+4k)+(2s+7k-5)=2s+11k-5$ rounds of the first and the second stages if $v(H')=0$.
    In both cases at the start of the third stage the host graph contains a blue $L^{(2)}_{t+s+k}$ and the wish triangle $T$.
    We use the wish triangle to perform \Cref{lem:gluepath2} joining two blue paths of $L^{(2)}_{t+s+k}$ into a single blue path within $k+4$ rounds.
    Hence we obtain a blue path on at least $t+s+k-k/2>t+s$ vertices.

    In total, the number of rounds in the game $\RR_F(C_k,P_{t+s})$ is upper bounded by $2s+11k-5+k+4 < 2s+12k$.
\end{proof}

\begin{lemma}\label{lem:extendblue}
    Suppose $H$ consists of two components: a blue edge $v_1v_2$ and an almost blue path $P_t$ on $t\ge 3$ vertices.
    Then Builder in the game $\RR_H(C_k,C_n)$ has a strategy such that
    \begin{itemize}
        \item either after at most 2 rounds there is an almost blue path on $t+2$ vertices,
        \item or after at most 4 rounds there is a colored graph with two components: a~wish triangle and an almost blue path on $t-2$ vertices.
    \end{itemize}
\end{lemma}
\begin{proof}
    If the colored path $P_t$ contains no red edge, then connect one of its endpoints to $v_1$.
    No matter what the color of the new edge is, we get an almost blue path on $t+2$ vertices.

    In the other case, let us denote by $u_1u_2$ the red edge of $P_t$.
    Builder selects edges $u_1v_1$ and $u_2v_2$, see \Cref{fig:almost_blue_path_extension} for depiction of the following case analysis.
    If one of the new edges is blue, then the path obtained from $P_t$ by replacing $u_1u_2$ with the path $(u_1,v_1,v_2,u_2)$ gives the result.
    If both are red, Builder selects edge $u_1v_2$ and, if $t\ge 4$, also the edge incident to two blue neighbors of $u_1$ and $u_2$ at the path $P_t$.
    If $u_1v_2$ is red, then $u_1u_2v_2$ constitute a red wish triangle; if it is blue then $u_1v_1v_2$ is a blue-blue-red wish triangle.
    In both cases we obtained a wish triangle and the almost blue path gets shorter by two vertices.
    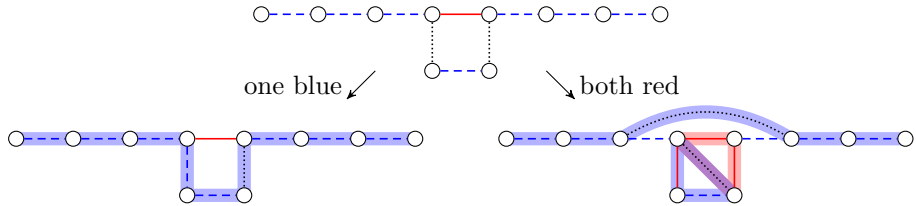
\begin{figure}[ht]
        \centering
        \begin{tikzpicture}[scale=.75]
            \begin{scope}
                \foreach \name/\x in {x1/1,x2/2,x3/3,u1/4,u2/5,x6/6,x7/7,x8/8}{
                    \node[main] (\name) at (\x,+1){};
                }
                \draw[blue] (x1) -- (x2) -- (x3) -- (u1) (u2) -- (x6) -- (x7) -- (x8);
                \draw[red] (u1) -- (u2);
                \node[main] (v1) at (4,0){};
                \node[main] (v2) at (5,0){};
                \draw[blue] (v1) -- (v2);
                \draw[tobe] (u1) -- (v1) (u2) -- (v2);
            \end{scope}
            \begin{scope}[shift={(-4.3cm,-2.2cm)}]
                \foreach \name/\x in {x1/1,x2/2,x3/3,u1/4,u2/5,x6/6,x7/7,x8/8}{
                    \node[main] (\name) at (\x,+1){};
                }
                \draw[blue] (x1) -- (x2) -- (x3) -- (u1) (u2) -- (x6) -- (x7) -- (x8);
                \draw[red] (u1) -- (u2);
                \node[main] (v1) at (4,0){};
                \node[main] (v2) at (5,0){};
                \draw[blue] (v1) -- (v2);
                \draw[blue] (u1) -- (v1);
                \draw[tobe] (u2) -- (v2);
                \draw[blue,selected] (x1) -- (x2) -- (x3) -- (u1) -- (v1) -- (v2) -- (u2) -- (x6) -- (x7) -- (x8);
            \end{scope}
            \begin{scope}[shift={(+4.3cm,-2.2cm)}]
                \foreach \name/\x in {x1/1,x2/2,x3/3,u1/4,u2/5,x6/6,x7/7,x8/8}{
                    \node[main] (\name) at (\x,+1){};
                }
                \draw[blue] (x1) -- (x2) -- (x3) -- (u1) (u2) -- (x6) -- (x7) -- (x8);
                \draw[red] (u1) -- (u2);
                \node[main] (v1) at (4,0){};
                \node[main] (v2) at (5,0){};
                \draw[blue] (v1) -- (v2);
                \draw[red] (u1) -- (v1);
                \draw[red] (u2) -- (v2);
                \draw[tobe] (u1) -- (v2);
                \draw[tobe] (x3) edge[bend left] (x6);
                \draw[red,selected] (u1) -- (u2) -- (v2) -- (u1);
                \draw[blue,selected] (u1) -- (v1) -- (v2) -- (u1);
                \draw[blue,selected] (x1) -- (x2) -- (x3) edge[bend left] (x6) (x6) -- (x7) -- (x8);
            \end{scope}
            \draw[->] (3,0) -- node[label=180:{one blue}]{} ++(-.5,-.5);
            \draw[->] (6,0) -- node[label=0:{both red}]{} ++(.5,-.5);
        \end{tikzpicture}
        \caption{
            Builder's strategy for extending an almost blue path by two or forcing a wish triangle.
            Final results of the strategy are shown by the outlined edges.
        }%
        \label{fig:almost_blue_path_extension}
    \end{figure}
\end{proof}

\begin{lemma}\label{lem:path2cycle}
    Suppose $k\ge 3$ is an odd integer and $H$ is an almost blue path on $t\ge 2k$ vertices.
    Let ${\mathcal C}$ be the family of all cycles of length at least $(t-k)/2$ and not greater than $t$.
    Then $\rr_H(C_k,{\mathcal C})\le k+1$.
\end{lemma}
\begin{proof}
    We assume that Painter never creates a red $C_k$.
    We divide the game into two stages but skip the first stage if $H$ is a blue path.

    In the first stage, for $H$ which is not blue, denote its red edge by $v_1v'_1$.
    There are two maximal blue paths $P$, $P'$ in $H$ so assume that $P$ has endpoints $v_1$ and $v_2$, while $v'_1$ and $v'_2$ are the endpoints of $P'$.
    We can assume that $P$ is not shorter than $P'$.
    Builder selects the edge $v'_1v_2$.
    If Painter colors it blue, we obtain a blue path $H'$ on the vertex set $V(H)$ and proceed to the second stage.
    Otherwise, both edges $v_1v'_1$ and $v'_1v_2$ are red.
    Then endings of $P$ can be through of as two paths $Q_1,Q_2$, on $\lfloor k/2\rfloor$ vertices each, such that their ends $v_1\in V(Q_1)$ and $v_2\in V(Q_2)$ are connected with a red path of length 2.
    Furthermore, $Q_1$ and $Q_2$ are disjoint since $v(Q_1)+v(Q_2)<k\le t/2\le |V(P)|$.
    Based on \Cref{lem:red_gluepath} within $k-2$ next rounds Builder can join them so that at most $\lfloor k/2\rfloor$ vertices of $Q\cup Q'$ are not contained in the resulting blue path $R$.
    Taking into account the other blue edges of $P$, we have a blue cycle on at least $v(P)-\lfloor k/2\rfloor\ge (t-k)/2$ vertices of $P$.
    The game ends in this case.

    In the second stage, we have a blue path $H'$ on $t$ vertices in the host graph (if $H$ was a blue path, then $H'=H$).
    Let $u_1,u_2,\dots,u_{t}$ be consecutive vertices of $H'$.
    Builder selects edges $u_1u_{\lfloor t/2\rfloor}$ and $u_{\lfloor t/2\rfloor}u_{t}$.
    If Painter colors one of them blue, then Builder gets a blue cycle on at least $\lfloor t/2\rfloor$ vertices and we are done.
    Otherwise, both of the edges are red so we apply \Cref{lem:red_gluepath} for two ending blue paths of $H'$ similarly as in the previous stage and conclude that within the next $k-2$ rounds Builder can force a blue cycle on at least $t-\lfloor k/2\rfloor$ vertices of $H'$.
    Summarizing, the second stage last at most $2+k-2=k$ rounds and results in a blue cycle of length at least $(t-k)/2$.

    It is not hard to verify that in all cases the game lasts at most $1+k$ rounds and the vertex set of the target blue cycle is contained in $V(H)$.
\end{proof}

Aside from an almost blue path Builder's strategy shall exploit another structure which we formally define next -- a blue path interlaced with every other vertex of a red path of twice the length.
\begin{definition}
    Let \emph{dragon tail} of length $m$ denoted by $\dtail(m)$ be a graph that consists of a red path of length $2m$ where odd vertices are connected into a blue path, see \Cref{fig:dragon_tail}.
    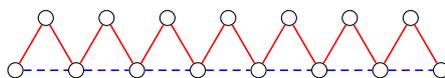
\begin{figure}[ht]
        \centering
        \begin{tikzpicture}[scale=.8]
            \tikzstyle{every node}=[main]
            \node (p1) at (0,0) {};
            \foreach \x in {2,...,8}{
                \pgfmathtruncatemacro\xx{\x-1}
                \node (p\x) at (\xx,0) {};
                \node (m\x) at ($(p\xx)+(60:1)$) {};
                \draw[blue] (p\xx) -- (p\x);
                \draw[red] (p\xx) -- (m\x) -- (p\x);
            }
        \end{tikzpicture}
        \caption{Dragon tail graph $\dtail(7)$}
        \label{fig:dragon_tail}
    \end{figure}
\end{definition}

\begin{lemma}\label{lem:redmatching}
    Suppose $H$ is a red matching of $m$ edges.
    Then Builder in the game $\RR_H(C_k,C_n)$ has a strategy such that for two starting vertices $s_1,s_2\notin V(H)$ one of the following holds.
    \begin{itemize}
        \item For some $t<m$, after at most $2t+6$ rounds there is a colored graph that contains vertex disjoint: two wish triangles and a blue path on $t$ vertices such that if $t\ge 2$, then the path starts in $s_1$ or $s_2$.
        \item After $2m$ rounds there is a dragon tail of length $m-2$ that starts in $s_1$ or $s_2$.
    \end{itemize}
\end{lemma}

\begin{proof}
    Builder gradually builds a dragon tail of length $d$.
    In every step, Builder will select two edges that either form a new wish triangle or they extend the dragon tail by one.

    Initially, the dragon tail has $0$ length and it starts in $s_1$.
    Having a partially built dragon tail of length $d$ in hand, let $u$ be its last vertex and let $vw$ be an unused red edge, i.e.,~$vw$ is an edge of $H$ such that Builder has not selected any edge incident to $v,w$ yet.
    Builder selects edges $uv$ and $uw$.

    Suppose that Painter colors edges $uv$ and $uw$ with different colors.
    Without loss of generality suppose that $uv$ is red and $uw$ is blue.
    Then the dragon tail is extended by one as the new blue edge $uw$ is adjacent to red edges $uv$ and $vw$.
    For the next step, Builder assigns $u:=w$, i.e., in the next step Builder selects the two edges from the vertex that was the endpoint of the blue edge, see \Cref{fig:dragon_tail2}.

    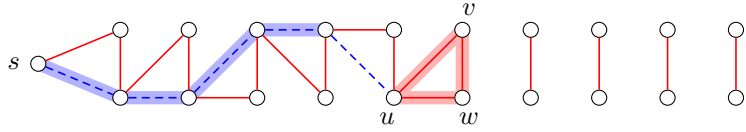
\begin{figure}
        \centering
        \begin{tikzpicture}[scale=0.9]
            \tikzstyle{every node}=[main]
            \begin{scope}
                \node[label=180:{$s$}] (w1) at (.8,.5) {};
                \foreach \x in {2,...,11}{
                    \node (v\x) at (\x,1) {};
                    \node (w\x) at (\x,0) {};
                    \draw[red] (v\x) -- (w\x);
                }
                \draw[blue] (w1) -- (w2) -- (w3) -- (v4) -- (v5) -- (w6);
                \draw[red] (w1) -- (v2);
                \draw[red] (w2) -- (v3);
                \draw[red] (w3) -- (w4);
                \draw[red] (v4) -- (w5);
                \draw[red] (v5) -- (v6);
                \draw[red] (v7) -- (w6) -- (w7);
                \draw[red,selected] (v7) -- (w6) -- (w7) -- (v7);
                \draw[blue,selected] (w1) -- (w2) -- (w3) -- (v4) -- (v5);
                \node[hide] at (5.9,-.3) {$u$};
                \node[hide] at (7.1,1.3) {$v$};
                \node[hide] at (7.1,-.3) {$w$};
            \end{scope}
        \end{tikzpicture}
        \caption{
            Initially red matching after six steps of Builder's strategy.
            The resulting blue path and the wish triangle are outlined.
            Observe the dragon tail $\dtail(5)$ spanning from $s$ to $u$.
        }
        \label{fig:dragon_tail2}
    \end{figure}

    If the selected edges $uv$ and $uw$ are of the same color, then a wish triangle $uvw$ was created.
    Builder adds this wish triangle to a set of wish triangles $T$ and shortens the dragon tail by one so that it does not overlap with the wish triangle.

    Let us remark that if the first wish triangle appears in the first step, then Builder chooses $s_2$ and starts building a dragon tail again.
    If the second wish triangle was created with $s_2$ as its vertex, then Builder finishes the game with the two wish triangles and any free vertex as the blue path on $t=0$ vertices.

    Builder continues increasing the dragon tail as long as there are less than two wish triangles and there are unused red edges of the matching.
    In the case where the second wish triangle is created at the moment the dragon tail has length $d\ge 1$, additionally to the two wish triangles Builder has a blue path $P$ of length $d-1$ as an induced subgraph of the dragon tail $\dtail(d)$ built so far and the path has one of its endpoints in $s_1$ or $s_2$.
    Notice that then the game lasts at most $2(d-1)+8$ rounds since Builder has selected $d-1$ pairs of edges incident to $P$ and additional 4 pairs of edges, two pairs for every wish triangle disjoint from a partially built dragon tail.

    After $2m$ rounds there are no unused red edges.
    If there is at most one wish triangle in the host graph, then Builder obtains a dragon tail of length $m-2$ (or longer if there is no with triangle).
\end{proof}

\newcommand{\builddragontail}[1]{
    \node[main] (p1) at (0,0){};
    \pgfmathtruncatemacro\lastVertex{#1+1}
    \foreach[count=\i] \x in {2,...,\lastVertex}{
        \node[main] (p\x) at (\i,0) {};
        \pgfmathtruncatemacro\xx{\x-1}
        \node[main] (m\x) at ($(p\xx)+(60:1)$) {};
        \draw[blue] (p\xx) -- (p\x);
        \draw[red] (p\xx) -- (m\x) -- (p\x);
    }
}
\newcommand{\builddragon}[3]{ 
    \begin{scope}[shift={(-#3 cm,0)}]
        \foreach \x in {1,...,#1}{
            \node[main] (c\x) at (360*\x/#1-360/#1:#3) {};
        }
        \foreach \f in {1,...,#1}{
            \pgfmathtruncatemacro\t{Mod(\f,#1)+1}
            \draw[blue] (c\f) -- (c\t);
        }
    \end{scope}
    \builddragontail{#2}
    \draw[blue] (p2) -- (c#1);
}

Having a dragon tail in hand, let us now show that Builder can force a long blue cycle.

\begin{lemma}\label{lem:tail2cycle}
    Suppose $H$ is a dragon tail of length $m$ and $u_0u_1$ is its first blue edge.
    Then Builder can play so that after less than $\log_2 m+8$ rounds of the game $\RR_H(C_k,C_n)$ there is a red $C_k$ or a blue cycle on more than $m-k$ but at most $m+1$ vertices of the dragon tail and contains the edge $u_0u_1$.
\end{lemma}

\begin{proof}
    We assume that Painter does not create a red $C_k$.
    Suppose that $m>k\ge 3$ and a dragon tail $H$ is a union of a blue path $u_0 u_1 \dots u_m$ and a red path $u_0w_1u_1\dots w_m u_m$.
    In the first round of $\RR_H(C_k,C_n)$ Builder selects the edge $u_0u_m$.
    If Painter colors it blue, then Builder reached his goal by obtaining a blue $C_{m+1}$.
    From now on, we assume that $u_0u_m$ is red.

    Note that if Builder created every edge from $u_0$ to $u_s$ for each $0 < s < m$, then for some index $i$ there would be a blue $u_0u_j$ and a red $u_0u_{j+1}$ because $u_0u_1$ is blue and $u_0u_m$ is red.
    However, in order to find such an index $j$, Builder can select much less edges: he performs a standard binary search (see \Cref{fig:dragon_tail_halving}).
    More precisely, Builder selects $u_0u_{\lfloor m/2\rfloor}$.
    If Painter colors it red, then Builder focuses on part of the path from $u_0$ to $u_{\lfloor m/2\rfloor}$; otherwise, he focuses on part of the path from $u_{\lfloor m/2\rfloor}$ to $u_m$.
    He then continues selecting edges $u_0u_i$ using the same approach on the focused part.
    In the end he ends up focusing on a part that contains only two vertices and therefore finds an index $1\le j\le m-1$ such that the edge $u_0u_j$ is blue and $u_0u_{j+1}$ is red.
    At each step the number of vertices between endpoints of the focused path is halved so this procedure takes at most $\lceil \log_2 m\rceil$ rounds.
    If $j \ge m-k$, then we have a blue cycle $u_0 u_1 \dots u_j$ on more than $m-k$ vertices.
    \begin{figure}[ht]
        \centering
        \def\tailsize{14}
        \begin{tikzpicture}[scale=0.7]
            \builddragontail{\tailsize}
            \draw[red]  (p1) edge[bend right=30]node[near end,above,sloped,inner sep=2pt,text=red]{1.}  (p15);
            \draw[red]  (p1) edge[bend right=32]node[very near end,below,sloped,inner sep=2pt,text=red]{2.}  (p8);
            \draw[blue] (p1) edge[bend right=36]node[near end,above,sloped,inner sep=2pt,text=blue]{3.} (p5);
            \draw[red]  (p1) edge[bend right=34]node[very near end,below,sloped,inner sep=2pt,text=red]{4.}  (p6);
            \draw[blue,selected] (p1) -- (p5) edge[bend left=36] (p1);
        \end{tikzpicture}
        \caption{
            Binary halving example on $\dtail(\tailsize)$ with order of edge placements by the Builder.
            If the outlined blue cycle is long enough, then Builder wins immediately.
        }
        \label{fig:dragon_tail_halving}
    \end{figure}
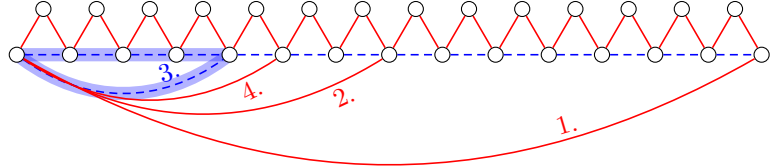

    From now on suppose that $j<m-k$.
    In the next round Builder selects the edge $u_{j+1}u_m$.
    We consider two possible Painter's responses.

    If $u_{j+1}u_m$ is red, then Builder achieves his goal by selecting $u_0u_s$ for $s=m-(k-3)/2$, see \Cref{fig:dragon_red_case}.
    If Painter colors $u_0u_s$ red, then Builder obtains a red cycle on $k$ vertices $u_s w_{s+1} u_{s+1} w_{s+2} \dots u_{m-1} w_m u_m u_{j+1} u_0$.
    If Painter colors it blue, Builder obtains a blue cycle on $m-(k-3)/2$ vertices $u_j u_{j+1}\dots u_s u_1 u_2 \dots u_{j-1}$ (note that this is more than $m-k$).
    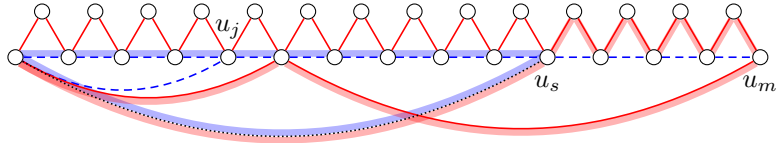
\begin{figure}[ht]
        \centering
        \def\tailsize{14}
        \begin{tikzpicture}[scale=0.7]
            \builddragontail{\tailsize}
            \node[label=90:$u_j$] at (p5) {};
            \node[label=-90:$u_s$] at (p11) {};
            \node[label=-90:$u_m$] at (p15) {};
            \draw[blue] (p1) edge[bend right] (p5);
            \draw[red]  (p1) edge[bend right] (p6);
            \draw[red]  (p6) edge[bend right] (p15);
            \draw[tobe] (p1) edge[bend right] (p11);
            \draw[blue,selected,transform canvas={shift={(0,.5mm)}},line width=2.5pt] (p1) edge[bend right,] (p11)(p11) -- (p1);
            \draw[red,selected,transform canvas={shift={(0,-.5mm)}},line width=2.5pt] (p1) edge[bend right] (p11)(p11) -- (m12) -- (p12) -- (m13) -- (p13) -- (m14) -- (p14) -- (m15) -- (p15) edge[bend left] (p6) (p6) edge[bend left] (p1);
        \end{tikzpicture}
        \caption{
            Cases where $u_{j+1}u_m$ is red within $\dtail(\tailsize)$ assuming $k=11$.
            Resulting cycles are outlined; the outline below edges is for the red cycle and the outline above edges is for the blue one.
        }
        \label{fig:dragon_red_case}
    \end{figure}

    Suppose that $u_{j+1}u_m$ is blue.
    We consider two cases based on the value of $k$ as we need to perform extra steps when $k$ is very small.

    For the first case, assume that $k\ge 7$.
    Builder achieves his goal by playing $u_{j-1}u_t$ with $t=j+(k-3)/2$, see \Cref{fig:dragon_blue_case}.
    If Painter colors $u_{j-1}u_t$ red, then we have a red cycle on $k$ vertices $u_{j-1} w_j u_j w_{j+1} \dots u_{t-1} w_t u_t$.
    If Painter colors it blue, then we have a blue cycle on at least $m-k$ vertices $u_0 u_1 \dots u_{j-1} u_t u_{t+1} \dots u_m u_{j+1} u_j$.
    \begin{figure}[ht]
        \centering
        \def\tailsize{14}
        \begin{tikzpicture}[scale=0.7]
            \builddragontail{\tailsize}
            \node[label=90:$u_j$] at (p5) {};
            \node[label=-90:$u_t$] at (p9) {};
            \node[label=-90:$u_m$] at (p15) {};
            \draw[blue] (p1) edge[bend right] (p5);
            \draw[red]  (p1) edge[bend right] (p6);
            \draw[blue] (p6) edge[bend right] (p15);
            \draw[tobe] (p4) edge[bend right] (p9);
            \draw[blue,selected,transform canvas={shift={(0,.5mm)}},line width=2.5pt] (p9) -- (p15) edge[bend left] (p6)(p6)--(p5)edge[bend left](p1)(p1) -- (p4) edge[bend right] (p9);
            \draw[red,selected,transform canvas={shift={(0,-.5mm)}},line width=2.5pt] (p4) -- (m5) -- (p5) -- (m6) -- (p6) -- (m7) -- (p7) -- (m8) -- (p8) -- (m9) -- (p9) edge[bend left] (p4);
        \end{tikzpicture}
        \caption{
            Cases where $u_{j+1}u_m$ is blue within $\dtail(\tailsize)$ assuming $k=11$.
            Resulting cycles are outlined; the outline below edges is for the red cycle and the outline above edges is for the blue one.
        }
        \label{fig:dragon_blue_case}
    \end{figure}
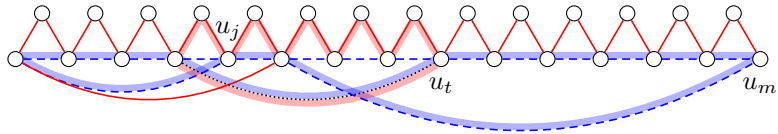

    In the other case where $k\le 5$, Builder always selects $u_{j-1}u_{j+2}$.
    If it is colored blue, then we have the same long blue cycle as in the first case.
    Hence, assume that the edge $u_{j-1} u_{j+2}$ is red.
    Builder proceeds based on the exact value of $k$, see \Cref{fig:dragon_small_cases}.
    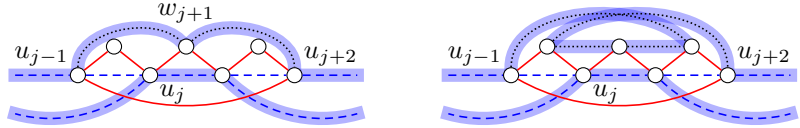
\begin{figure}[ht]
        \def\drawsmallbase{
            \node[main] (p1) at (0,0) {};
            \foreach \x in {2,...,4}{
                \pgfmathtruncatemacro\xx{\x-1}
                \node[main] (p\x) at (\xx,0) {};
                \node[main] (m\x) at ($(p\xx)!.5!(p\x)+(0,.4)$) {};
                \draw[blue] (p\xx) -- (p\x);
                \draw[red] (p\xx) -- (m\x) -- (p\x);
            }
            \draw[red] (p1) edge[bend right] (p4);
            \node (p0) at (-1,0) {};
            \draw[blue] (p0) -- (p1);
            \node (p5) at (4,0) {};
            \draw[blue] (p4) -- (p5);
            \node (leftbot) at (-1,-.5) {};
            \draw[blue] (p2) edge[bend left] (leftbot);
            \node (rightbot) at (4,-.5) {};
            \draw[blue] (p3) edge[bend right] (rightbot);
            \node[anchor=south east] at (p1) {$u_{j-1}$};
            \node[anchor=south west] at (p4) {$u_{j+2}$};
            \node[anchor=north west] at (p2) {$u_j$};
        }
        \centering
        \begin{tikzpicture}[scale=.95,looseness=.9]
            \begin{scope}
                \drawsmallbase
                \node[anchor=south] at ($(m3)+(0,.2)$) {$w_{j+1}$};
                \draw[tobe] (p1) edge[out=90,in=135] (m3);
                \draw[blue,selected] (p1) edge[out=90,in=135] (m3);
                \draw[tobe] (m3) edge[out=45,in=90] (p4);
                \draw[blue,selected] (m3) edge[out=45,in=90] (p4);
                \draw[blue,selected] (p0) -- (p1) (p4) -- (p5);
                \draw[blue,selected] (rightbot) edge[bend left] (p3) (p3) -- (p2) edge[bend left] (leftbot);
            \end{scope}
            \begin{scope}[shift={(6cm,0)}]
                \drawsmallbase
                \draw[tobe] (p1) edge[out=90,in=135] (m4);
                \draw[blue,selected] (p1) edge[out=90,in=135] (m4);
                \draw[tobe] (m2) edge[out=45,in=90] (p4);
                \draw[blue,selected] (m2) edge[out=45,in=90] (p4);
                \draw[tobe] (m2) -- (m3) -- (m4);
                \draw[blue,selected] (m2) -- (m3) -- (m4);
                \draw[blue,selected] (p0) -- (p1) (p4) -- (p5);
                \draw[blue,selected] (rightbot) edge[bend left] (p3) (p3) -- (p2) edge[bend left] (leftbot);
            \end{scope}
        \end{tikzpicture}
        \caption{
            Closeup of vertices from $u_{j-1}$ to $u_{j+2}$ of the dragon tail; we see the final set of cases where $u_{j+1}u_m$ is blue for $k=5$ (left) and $k=3$ (right).
            Note how the created (dotted) edges form a red cycle of length $k$ if Painter colors them red.
            The final blue cycle is outlined.
        }%
        \label{fig:dragon_small_cases}
    \end{figure}

    If $k=5$, then in the next two rounds Builder forces edges $u_{j-1}w_{j+1}$ and $w_{j+1}u_{j+2}$ blue and hence we get a blue cycle longer than $m-k$, on vertices
    \[
        u_0 u_1 \dots u_{j-1} w_{j+1} u_{j+2} u_{j+3} \dots u_m u_{j+1} u_j.
    \]
    If $k=3$, then in the next four rounds Builder forces a blue path $u_{j+2}w_{j}w_{j+1}w_{j+2}w_{j-1}$ and again a cycle longer than $m-k$ arises on vertices
    \[
        u_0 u_1 \dots u_{j-1} w_j w_{j+1} w_{j+2} u_{j+2} u_{j+3} \dots u_m u_{j+1} u_j.
    \]
    Finally, observe that in all cases Builder achieves his goal within at most $1+\lceil \log_2 m\rceil+6$ rounds of the game $\RR_H(C_k,C_n)$.
\end{proof}

\begin{definition}
    Let \emph{dragon graph} $\dragon(t,m)$ be a graph that consists of a blue cycle on $t$ vertices and a dragon tail $\dtail(m)$ of length $m$ that shares one of its endpoints $u$ with the blue cycle, and has a blue edge between vertex of the cycle adjacent to $u$ and a vertex of the dragon tail adjacent via blue edge with $u$, see \Cref{fig:dragon}.
\end{definition}

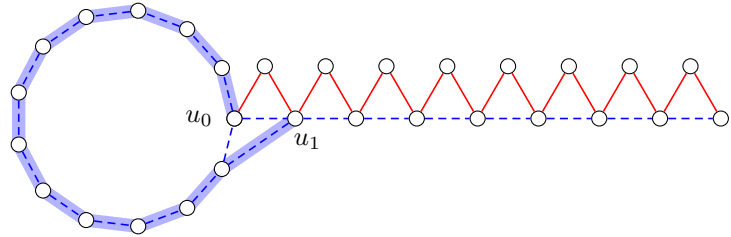
\begin{figure}[ht]
    \centering
    \def\cyclesize{13}
    \def\tailsize{8}
    \begin{tikzpicture}[scale=0.8]
        \builddragon{\cyclesize}{\tailsize}{1.8}
        \draw[blue,selected] (c1) -- (c2) -- (c3) -- (c4) -- (c5) -- (c6) -- (c7) -- (c8) -- (c9) -- (c10) -- (c11) -- (c12) -- (c13) -- (p2);
        \node[anchor=east] at ($(p1)+(-.2,0)$) {$u_0$};
        \node[anchor=north] at ($(p2)+(.2,-.1)$) {$u_1$};
    \end{tikzpicture}
    \caption{Dragon graph $\dragon(\cyclesize,\tailsize)$ with outlined path that is used instead of edge $u_0u_1$.}
    \label{fig:dragon}
\end{figure}

\begin{lemma}\label{lem:blueredending}
    Suppose $H$ is a dragon graph $\dragon(t,m)$, $t\ge 3$ and $m>k\ge 3$.
    Let ${\mathcal C}$ be the family of all cycles of length at least $t+m-k$ and not greater than $t+m$.
    Then
    \[
        \rr_{\dragon(t,m)}(C_k,{\mathcal C})\le \log_2 m +8.
    \]
\end{lemma}

\begin{proof}
    Let $u_0u_1$ be the first edge of the tail of the dragon graph.
    Suppose that $u_0$ lies on the blue cycle $C$ of the dragon graph, $x$ is its neighbor at the cycle and the edge $xu_1$ is blue.
    Based on Lemma \ref{lem:tail2cycle}, Builder can play so that after at most $\log_2 m+8$ rounds of the game $\RR_H(C_k,C_n)$ there is a blue cycle $C'$ on more than $m-k$ vertices of the dragon tail and $u_0u_1\in E(C')$, see \Cref{fig:dragon}.
    Observe that $(C\cup C'\cup\{xu_1\})\setminus\{u_0u_1\}$ is a blue cycle on at least $t+m-k$ vertices of $H$.
\end{proof}

\subsection{Builder's strategy}

For simplicity, assume that Painter never creates a red $C_k$.
Builder's strategy for odd $k$ is presented in three separate stages.
The first stage ends if either the colored graph has $2n+8k$ vertices or it contains two wish triangles.
In the second stage, Builder creates a blue cycle on at least $n$ vertices.
In the third stage, Builder shortens the blue cycle to exactly $n$ vertices.

\subsubsection*{Stage I}

Throughout the first stage, Builder keeps three vertex disjoint structures: a red matching $M$, an almost blue path $B$, and a set of wish triangles $T$.
As long as the host graph has less than $2n+8k$ vertices and $T$ contains less than two triangles, Builder draws a free edge.
If it is colored red, he adds it to $M$.
Otherwise, he plays according to \Cref{lem:extendblue} which either extends $B$, or creates a wish triangle and shortens $B$ by two vertices.
In case the wish triangle was created, Builder puts it to $T$ and if $T$ contains two triangles, the first stage ends.
See the decision diagram in \Cref{fig:stagesCases}.

Suppose that at the end of Stage I we have $v(B)=t$, $e(M)=m$ and the host graph is $H_1$.
Based on \Cref{lem:extendblue}, one can calculate that the number of rounds in Stage I is not greater than
\[
    e(M)+\frac32 (v(B)+4)+10=m+\frac32 t+16.
\]
The first term stands for one round used for each edge in $M$, second term represents three edges created for each pair of vertices within $B$ (plus $4$ because of the lost vertices when wish triangles are created), and the last term is for the $5$ rounds spent creating each wish triangle.

The number of vertices of the host graph at the end of Stage I satisfies the inequalities
\begin{equation}\label{eq:first_stage}
    v(H_1)\le 2n+8k\text{ and }2m+t\le v(H_1)\le 2m+t+8.
\end{equation}
Furthermore, $v(H_1)= 2n+8k$ if $T$ contains less than two wish triangles.

\subsubsection*{Stage II}
This stage starts with the host graph $H_1$, containing $M$, $B$, and $T$ from Stage I.
The goal is to obtain a blue cycle of length between $n$ and $2n+\Oh(k)$.

We consider three cases, the first one is based on appearance of two wish triangles.
In the other cases, only an almost blue path $B$ on $t$ vertices and a red matching $M$ with $m$ edges is important for Builder.
Then, in view of the calculations at the end of Stage I, we have $t+2m\ge v(H_1)-8 = 2n+8k-8$, hence either $M$ is big or $B$ is long -- this constitutes the second and the third case of this stage.
See \Cref{fig:stagesCases} for an overview of how the cases interact.

\begin{figure}[!ht]
    \centering
    \begin{tikzpicture}[scale=.8]
        \tikzset{decision/.style={black, fill=white, chamfered rectangle, chamfered rectangle xsep=1mm}}
        \tikzset{structure/.style={violet, draw, fill=white, rectangle, rounded corners=2mm}}
        \tikzset{action/.style={black, fill=white, draw, rectangle, thick}}
        \def\stagewidth{15cm}
        \newcommand{\prounds}[1]{{\small\color{green!60!black}(#1)}}
        \begin{scope}[shift={(0cm,13cm)}] 
            \node[text=gray,rotate=90] at (.5,4.5) {Stage I};
            \node[structure] at (8,8) {$B \cup M \cup T$};
            \node[decision] (loop_stage1) at (8,7) {Less than $2n+8k$ vertices in the colored graph?};
            \node[action,text width=2.5cm] (s1_edge) at (7,5) {Create a free edge $e$ \prounds{1 round}};
            \node[decision] (s1_edge_colored) at (6,3) {$e$ is colored};
            \node[action] (s1_red) at (10,3) {Add $e$ to $M$};
            \node[action,text width=3cm] (s1_blue) at (4,1) {\Cref{lem:extendblue} \prounds{$8$ per $\TRI$ \& $1$ per path vertex}};
            \node[decision] (s1_blue_path_long) at (12,1) {Is blue path long?};
            \draw[thick,gray] (0,0) -- ++(\stagewidth,0);
            \draw (loop_stage1) edge[->,bend right=10]node[right]{yes} (s1_edge);
            \draw (loop_stage1.-10) edge[->,bend left=30]node[right]{no} (s1_blue_path_long);
            \draw (s1_edge) edge[->,bend left=10] (s1_edge_colored);
            \draw (s1_edge_colored) edge[red,->,bend right=50]node[text=red,above]{red} (s1_red);
            \draw (s1_edge_colored) edge[blue,->,bend left=10]node[text=blue,left]{blue} (s1_blue);
            \draw (s1_red) edge[->,bend right=15] (loop_stage1);
            \draw (s1_blue) edge[->,bend left] node[text width=2cm,left]{longer $B$ or $1^{\rm st}$ wish triangle in $T$} (loop_stage1.-170);
        \end{scope}
        \begin{scope}[shift={(0cm,6cm)}] 
            \node[text=gray,rotate=90] at (.5,3.5) {Stage II};
            \draw[thick,gray] (0,0) -- ++(\stagewidth,0);
            \node[structure] (approximate_cycle) at (8.2,0) {$C_{n+\Oh(k)}$};
            \begin{scope}[shift={(0.0cm,6.8cm)},text=gray,draw=gray,anchor=north] 
                \node (a) at (4,0) {wish triangles};
                \node (b) at (8.5,0) {red matching};
                \node (c) at (13,0) {blue path};
                \draw ($(a)!.5!(b)$) -- ++(0,-6);
                \draw ($(c)!.5!(b)$) -- ++(0,-6);
            \end{scope}
            \begin{scope}[shift={(4cm,0cm)}] 
                \node[structure] (s2_wishes_and_path) at (0,5.5) {$2\TRI \cup\, P_{t'}$};
                \node[action] (wish_blue_path) at (0,4.5) {\Cref{lem:longerpath} \prounds{$2s+12k$}};
                \node[structure] (s2_wish_and_path) at (0,3.5) {$\TRI \cup P_n$};
                \node[action] (wish_join_cycle) at (0,2.5) {\Cref{lem:gluepath} \prounds{$k+4$}};
                \draw[->] (s2_wishes_and_path) -- (wish_blue_path);
                \draw[->] (wish_blue_path) -- (s2_wish_and_path);
                \draw[->] (s2_wish_and_path) -- (wish_join_cycle);
            \end{scope}
            \begin{scope}[shift={(8.2cm,0cm)}] 
                \node[structure] (s2_matching) at (0,5.5) {$M$};
                \node[action] (match_tail) at (0,3.5) {\Cref{lem:redmatching} \prounds{$2m$ or $2t'+6$}};
                \node[structure] (s2_tail) at (0,2.5) {$\dtail(m-\Oh(1))$ tail};
                \node[action] (match_tail_search) at (0,1.5) {\Cref{lem:tail2cycle} \prounds{$\log_2m+8$}};
                \draw[->] (s2_matching) -- (match_tail);
                \draw[->] (match_tail) -- (s2_tail);
                \draw (match_tail) edge[->,out=125,in=-5] (s2_wishes_and_path);
                \draw[->] (s2_tail) -- (match_tail_search);
            \end{scope}
            \begin{scope}[shift={(13cm,0cm)}] 
                \node[structure] (s2_path) at (0,5.5) {blue path $B$};
                \node[action] (s2_to_cycle) at (0,4.5) {\Cref{lem:path2cycle} \prounds{$k+1$}};
                \node[structure] (s2_small_cycle) at (0,3.5) {$C_{<n}$};
                \node[action] (s2_build_dragon) at (0,2.5) {\Cref{lem:redmatching} \prounds{$2m$ or $2t'+6$}};
                \node[structure] (s2_dragon) at (0,1.5) {$\dragon(t',m-\Oh(1))$};
                \node[action] (path_case4) at (0,0.5) {\Cref{lem:blueredending} \prounds{$\log_2m+8$}};
                \draw[->] (s2_path) -- (s2_to_cycle);
                \draw[->] (s2_to_cycle) -- (s2_small_cycle);
                \draw (s2_to_cycle) edge[->,out=-135,in=45] (approximate_cycle);
                \draw[->] (s2_small_cycle) -- (s2_build_dragon);
                \draw (s2_build_dragon) edge[->,out=135,in=-5] (s2_wishes_and_path);
                \draw[->] (s2_build_dragon) -- (s2_dragon);
                \draw[->] (s2_dragon) -- (path_case4);
            \end{scope}
            \draw (wish_join_cycle) edge[->,out=-80,in=135] (approximate_cycle);
            \draw[->] (match_tail_search) -- (approximate_cycle);
            \draw[->] (path_case4) edge[out=184,in=25] (approximate_cycle);
        \end{scope}
        \draw (s1_blue) edge[->,bend right=20] (s2_wishes_and_path);
        \draw (s1_blue_path_long) edge[->,bend left=12] (s2_matching);
        \draw (s1_blue_path_long) edge[->,bend right=30] (s2_path);
        \begin{scope}[shift={(0cm,0.5cm)}] 
            \node[text=gray,rotate=90] at (.5,4.2) {Stage III};
            \node[action] (s3_theorem) at (8.2,4.5) {\Cref{thm:shortenall} \prounds{$3k+10$}};
            \node[structure] (s3_cycle) at (8.2,3.5) {$C_n$};
            \draw[->] (approximate_cycle) -- (s3_theorem);
            \draw[->] (s3_theorem) -- (s3_cycle);
        \end{scope}
    \end{tikzpicture}
    \caption{
        Overview of Builder's strategy in Stage I, Stage II, and Stage III.
        $\TRI$ denotes wish triang
        Thick rectangles signify Builder's steps (with number of rounds in brackets); angled corners are decisions; purple boxes with rounded corners are structures that Builder obtained in that particular phase of the strategy.
        Arrows from lemmata point to graphs that may be obtained by them.
    }%
    \label{fig:stagesCases}
\end{figure}
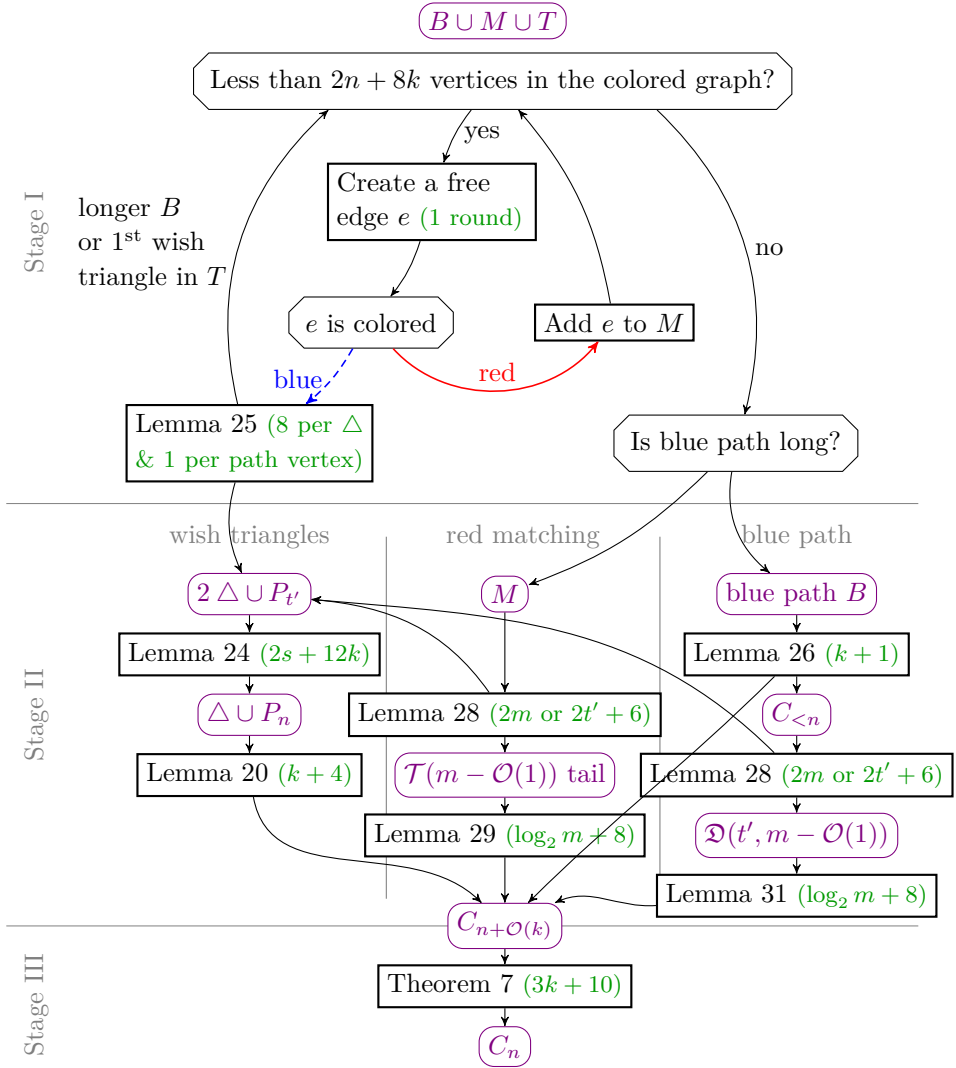

\paragraph*{Two wish triangles Case}
In this case Builder focuses on two wish triangles in $T$ and the almost blue path $B$ on $t\le 2n+8k$ vertices ($t=1$ for an empty path).
If $t> n+\lfloor k/2\rfloor$, then we define $B'$ as an almost blue path on $t'=n+\lfloor k/2\rfloor$ vertices, contained in $B$; otherwise we put $B'=B$ and $t'=t$.
We apply \Cref{lem:longerpath} for $s=n-t'+\lfloor k/2\rfloor$, the path $B'$ and the first of the wish triangles from $T$.
Thereby Builder gets a blue path $B''$ on $t'+s=n+\lfloor k/2\rfloor$ vertices, within $2s+12k\le 2n-2t'+13k$ rounds.
Then, having the blue path $B''$ and the second wish triangle from $T$, in view of \Cref{lem:gluepath2cycle} Builder can force a blue cycle of length between $n$ and $n+\lfloor k/2\rfloor$, within $k+4$ next rounds.
It ends Stage II.
The maximum number of rounds of the second stage in Two wish triangles Case is at most
$(2n-2t'+13k)+(k+4)=2n-2t'+14k+4$.

Other cases may end up having this case as a subroutine so the above number will be considered in analysis for the total number of rounds.

In order to calculate the total number of rounds, observe that the definition of $t'$ implies that $t-2t'<0$ for $t>n+\lfloor k/2\rfloor$, while for $t\le n+\lfloor k/2\rfloor$ we have
\begin{equation}\label{eq:eight_k}
    t-2t'\le v(H_1)-2(n+\lfloor k/2\rfloor)<8k.
\end{equation}
Now we are ready to bound the number of rounds in two stages in Two wish triangles Case.
These number is not greater than
\[
    (m+\frac32 t+16)+(2n-2t'+14k+4) \putabove{\le}{(\ref{eq:first_stage})} \frac12v(H_1)+t-2t'+2n+14k+20\putabove{\le}{(\ref{eq:first_stage},\ref{eq:eight_k})} 3n+26k+20.
\]

\paragraph*{Red matching Case}
In this case we assume that there are less than two wish triangles in $T$ and the almost blue path $B$ is short, namely $t<2k$.
Then Builder ignores $T$ and $B$, focuses on $M$ and plays according to \Cref{lem:redmatching} choosing as starting points any two free vertices.
Then two situations can happen.

\begin{enumerate}
    \item
        For some $t'<m$, after at most $2t'+6$ rounds, we have a blue path on $t'$ vertices and two wish triangles.
        Then Builder continues as in Two wish triangle Case
        and creates a blue cycle of length between $n$ and $n+\lfloor k/2\rfloor$.

        The number of rounds of Stage II in this case is at most $(2t'+6)+(2n-2t'+14k+4)=2n+14k+10$.
        Together with Stage I, for $t<2k$ we have at most
        \[
            (m+\frac32 t+16)+(2n+14k+10)\le \frac12 v(H_1)+t+2n+14k+26<3n+20k+26
        \]
        rounds.

    \item
        After $2m$ rounds we have a dragon tail of length $m-2$.
        Then Builder plays as in \Cref{lem:tail2cycle} and creates a blue cycle of length between $m-2-k$ and $m-1$, within less than $\log_2 m+8$ rounds.
        Observe that $m-1\le(v(H_1)-t)/2\le n+4k$ and $m-2-k\ge (v(H_1)-8-t)/2-2-k> v(H_1)/2-2k-6\ge n$.
        Thus we have a blue cycle of length between $n$ and $n+4k$ and the number of rounds of Stage II in this case is less than $2m+\log_2 m+8$.
        The game proceeds to Stage III.

        The number of rounds of Stages I and II in this case is at most
        \[
            (m+\frac32 t+16)+(2m+\log_2 m+8)\le \frac32 v(H_1)+\log_2 m+24\le 3n+\log_2 n+12k+26.
        \]
\end{enumerate}

\paragraph*{Almost blue path Case}
In this case we assume that there are less than two wish triangles in $T$ and the almost blue path $B$ has at least $2k$ vertices.
Then Builder first focuses on $B$ and based on \Cref{lem:path2cycle} he forces a blue cycle $C$ such that $(t-k)/2\le v(C)\le t$, within at most $k+1$ rounds.
If $v(C)\ge n$, then we end Stage II.
Otherwise, Builder applies his strategy from \Cref{lem:redmatching}, using as starting vertices two adjacent vertices $s_1,s_2$ of the blue cycle $C$.
Furthermore, he never selects edges incident to vertices of $V(C)\setminus \{s_1,s_2\}$.
This results in two subcases.
\begin{enumerate}
    \item
        For some $s<m$ after $2s+6$ rounds the part of the host graph disjoint from $V(C)\setminus \{s_1,s_2\}$ contains three vertex disjoint graphs: two wish triangles and a blue path $B'$ on $s$ vertices such that if $s\ge 2$, then $B'$ starts in $s_1$ or $s_2$.

        If $s=1$, then we have two wish triangles and a blue path on the vertex set $V(C)\setminus\{s_1,s_2\}$, disjoint from the wish triangles.
        If $s\ge 2$, then $B'$ starts in $s_1$ or $s_2$, say $s_2$.
        Then $s_1$ is the only vertex of $C$ which may be a vertex of a wish triangle.
        Thus the sum of $B'$ and the blue path $C\setminus\{s_1\}$ is a blue path on $t'=s+v(C)-1$ vertices.
        Further Builder proceeds as in Two wish triangles Case and creates a blue cycle of length between $n$ and $n+\lfloor k/2\rfloor$.

        The number of rounds of Stage II in this case is at most
        \[
            (k+1)+(2s+6)+(2n-2t'+14k+4)=2n-2v(C)+15k+13\le 2n-t+16k+13.
        \]
        The number of rounds of Stages I and II is
        \[
            (m+\frac32 t+16)+(2n-t+16k+13)\le \frac12 v(H_1)+2n+16k+29\le 3n+20k+29.
        \]

    \item
        After $2m$ rounds Builder obtains a dragon $\dragon(v(C),m-2)$.
        Then we can apply \Cref{lem:blueredending} to that dragon, since
        \begin{eqnarray*}
            m-2 \ge  \frac12(v(H_1)-8-v(B))-2&\ge& \frac12(v(H_1)-2v(C)-k)-6\\
                                             &\ge& \frac12(v(H_1)-2n-k)-10>k.
        \end{eqnarray*}
        Because the dragon has $v(C)+m-2$ vertices and
        \[
            2n+8k>v(C)+m-2\ge \frac12(t-k)+m-2\ge\frac12(v(H_1)-8-k)-2> n,
        \]
        Builder gets a blue cycle of length between $n+1$ and $2n+8k$ and it takes him less than $\log_2 m + 8$ next rounds.
         The number of rounds of Stage II in this case is at most
        \[
            (k+1)+2m+(\log_2 m + 8)=2m+\log_2 m+k+9.
        \]
        Together with Stage I it gives
        \begin{eqnarray*}
            (m+\frac32 t+16)+(2m+\log_2 m+k+9)&\le& \frac32 v(H_1)+\log_2 m+k+25\\
                                              &\le& 3n+\log_2 n+13k+27
        \end{eqnarray*}
        rounds.
\end{enumerate}

\subsubsection*{Stage III}

Builder now has a blue cycle of length between $n$ and $2n+8k$.
For $n\ge 8k$ we apply \Cref{thm:shortenall} to get
a blue $C_n$, which finishes the game.
The number of rounds in Stage III is at most $3k+10$.

Summarizing, the total number of rounds in the game is less than
\[
    3n+\log_2 n+26k+27+3k+10< 3n+\log_2 n+50k.
\]

\section{Short odd cycle vs long path}

It remains to prove \Cref{thm:oddpath}.
The proof is similar to the argument in the previous section, however, it is simpler because only one wish triangle is enough for Builder's need and we do not have to perform the binary search at a dragon tail.
Thus we repeat Stage I, then in Stage II Two wish triangles Case ends when the blue path $B''$ is built.
As for Red matching Case and Almost blue path Case, we repeat the argument but we finish immediately after creating a path of length at least $n$.
We do not need Stage III.
Furthermore, we do need the assumption that $n\ge 8k$ since it is used in Stage III only.
Summarizing, after some rough calculations, we obtain that Builder can force a red $C_k$ or a blue $P_n$ within less than $3n+40k$ rounds.

\section{Future work}

We have proved that $\rr(C_{k},P_n)=2n+o(n)$ for every fixed even $k\ge 3$.
We have no such asymptotically tight result in case of odd $k$.
However, our results, as well as the lower bound on $\rr(C_k,C_n)$, suggest that for fixed odd $k$ the online Ramsey number $\rr(C_k,C_n)$ is not far from $3n$.
We conjecture that it is $(3+o(1))n$.

\begin{conjecture}
    $\rr(C_{k},C_n)=3n+o(n)$ for every fixed odd $k\ge 3$.
\end{conjecture}

Another question is how the number of the host graph affects the number of rounds in the game.
One can verify that in the proof of \Cref{thm:even} and \Cref{thm:odd} Builder can keep the order of the host graph not greater than $n+ck$ for even $k$, and $2n+ck$ for odd $k$, where $c$ is some absolute (and not big) constant.
For big $n$ these two numbers are quite close to the corresponding Ramsey numbers $r(C_k,C_n)$, equal to $n+k/2-1$ and $2n-1$, respectively.
It would be interesting to know what happens in a restricted version of the Ramsey game, when Builder is allowed to keep the order of the host graph not greater than $r(C_k,C_n)$.
Clearly, the number of rounds in such a game is bounded by the restricted size Ramsey number $r^*(C_k,C_n)$, but it is known only for odd $k$ and is quadratic in $n$.
More precisely, $r^*(C_k,C_n)=\lceil (2n-1)(n+1)/2 \rceil$ for any fixed odd $k\ge 3$ and big enough $n$, as shown by \L{}uczak, Polcyn and Rahimi \cite{LPR}.

\siamver{
\bibliographystyle{siamplain}
}{
\bibliographystyle{plainurl}
}
\bibliography{main}

\end{document}